\documentclass[leqno,11pt]{article}%
\usepackage{amsfonts}
\usepackage{amsmath}
\usepackage{amssymb}
\usepackage{palatino}
\usepackage[colorlinks]{hyperref}
\usepackage{graphicx}%
\providecommand{\U}[1]{\protect\rule{.1in}{.1in}}
\hypersetup{colorlinks=true, urlcolor=red, linkcolor=blue, citecolor=blue}
\newtheorem{theorem}{Theorem}[section]
\newenvironment{acknowledgement}
{\par\medskip\noindent\textbf{Acknowledgement.}\ }
{\par\medskip}

\newtheorem{definition}{Definition}[section]
\newtheorem{example}{Example}[section]

\newtheorem{lemma}{Lemma}[section]

\newtheorem{remark}{Remark}[section]

\newenvironment{proof}[1][Proof]{\noindent\textbf{#1.} }{\ \rule{0.5em}{0.5em}}
\begin{document}

\title{Viability problems for SDEs in a general Gelfand triple setup using a
variational approach~\thanks{E-mails:{\small ioana.ciotir@insa-rouen.fr (Ioana
Ciotir), aurika\_00@yahoo.com (Aureliu Ionescu), eduard.rotenstein@uaic.ro
(Eduard Rotenstein, corresponding author)}}}
\author{Ioana Ciotir$^{a}$, Aureliu Ionescu$^{b}$, Eduard Rotenstein$^{b,c}$\medskip\\$^{a}${\small Normandie University, INSA de Rouen Normandie, LMI (EA 3226 --
FR CNRS 3335),}\\{\small 76000 Rouen, France}\\$^{b}${\small Simion Stoilow Institute of Mathematics of the Romanian Academy,
Bucharest, Romania}\\$^{c}${\small Faculty of Mathematics, "Alexandru Ioan Cuza" University, 9
Carol I Blvd., Ia\c{s}i,} {\small Romania}}
\maketitle

\begin{abstract}
We establish viability conditions for a broad general class of stochastic
differential equations formulated within the variational approach,%
\[
dX_{t}=A\left(  X_{t}\right)  dt+B\left(  X_{t}\right)  dW_{t}.
\]
The underlying functional space setup is given by a general Gelfand triple
$V\subset H\subset V^{\ast},$ where $H$ is a separable Hilbert space in which
the stochastic evolution equation (SEE) is considered. Our approach to
deriving viability conditions is inspired by the techniques developed by Aubin
and Da Prato for finite-dimensional forward stochastic differential equations.
We establish necessary and sufficient conditions for the viability of closed
random constraint sets in terms of adapted variational tangent (and
contingent) sets. Furthermore, we illustrate the applicability of the proposed
framework by discussing several important classes of stochastic differential
equations that fit naturally within our variational setting.

\end{abstract}



\section{Introduction}

\subsection{State of the art}

The theory of stochastic evolution equations has become one of the central
frameworks for modeling systems influenced by randomness in
infinite-dimensional state spaces. Such equations arise naturally in fluid
mechanics, diffusion processes, population dynamics, image analysis,
mathematical finance, and control, where the underlying state is often
distributed in space and the random perturbation is best described by a
cylindrical Wiener process. For this reason, stochastic partial differential
equations (SPDEs) are now studied through several complementary approaches,
among which the variational formulation occupies a distinguished position
because it allows one to treat nonlinear and dissipative systems in a rigorous
functional-analytic setting. In the standard theory developed by
Pr\'{e}v\^{o}t and R\"{o}ckner, the state space is formulated as a Gelfand
triple and existence and uniqueness are derived under hemicontinuity,
monotonicity, coercivity, and boundedness assumptions on the drift and
diffusion coefficients Pr\'{e}vot, R\"{o}ckner \cite{Prevot/Rockner:07}, Da
Prato, Zabczyk \cite{DaPrato/Zabczyk:14}, Liu, R\"{o}ckner
\cite{Liu/Rockner:15}. This framework has proved very effective for
establishing well-posedness, yet it does not by itself address the geometric
question of whether the stochastic evolution remains inside a prescribed
constraint set.

Constraint preservation, or viability, is a fundamentally different issue. In
deterministic and stochastic dynamics alike, one is often interested not
merely in the existence of a solution, but in the persistence of
admissibility: positivity, boundedness, state constraints, obstacle-type
restrictions, or more general closed convex or non-convex conditions. The
conceptual foundation for this line of research was laid by Aubin, Da Prato
(\cite{Aubin/DaPrato:90}, \cite{Aubin:09}, \cite{Aubin:90}) who extended
Nagumo-type viability ideas to stochastic differential equations and
introduced stochastic contingent and tangent sets for random closed
constraints. Their work remains the classical reference for stochastic
viability and invariance, and it established the guiding principle that
viability can be characterized by local geometric compatibility between the
drift, the diffusion, and the boundary structure of the constraint set.

Stochastic viability for differential equation driven by fractional Brownian
motion was first introduced by Ciotir and R\u{a}\c{s}canu in
\cite{Ciotir/Rascanu:09} and it was latter extended in several directions (see
e.g. \cite{XuLuo}, \cite{Zhu}).

Related developments in stochastic control and moving-set problems further
clarified how random and time-dependent constraints can be handled in
stochastic systems, especially in finite dimensional or special
infinite-dimensional settings. For instance, Buckdahn, Quincampoix, Rainer,
R\u{a}\c{s}canu (\cite{Buckdahn/Quincampoix/Rainer/Rascanu:02},
\cite{Buckdahn/Quincampoix/Rascanu:00}) approached the viability issue, for
both forward and backward SDEs via the viscosity solutions for some suitable
chosen PDEs. However, the original results were developed in a finite
dimensional context, and the corresponding infinite dimensional theory has
evolved only partially in parallel with the general theory of SPDEs.

In parallel, the classical variational literature has continued to intensively
develop through stochastic porous media equations (see e.g. \cite{Barbu:2010},
\cite{Barbu/DaPrato:06}, \cite{Barbu/DaPrato/Rockner:09},
\cite{Barbu/Rockner:13}, \cite{Ciotir:11}, \cite{Ciotir/Tolle:16},
\cite{Ciotir:17}, \cite{Ciotir/Fukuizumi/Goreac:25},
\cite{Ciotir/Goreac/Li/Peng/Rotenstein:26}, stochastic variational
inequalities (see e.g. \cite{Rascanu/Rotenstein:11},
\cite{Rascanu/Rotenstein:17}, \cite{Gassous/Rascanu/Rotenstein:12},
\cite{Gassous/Rascanu/Rotenstein:15}), stochastic Stefan problem (see e.g.
\cite{Barbu/DaPrato:2002}, \cite{Barbu/Ciotir/Danaila:21},
\cite{Ciotir/Goreac/Flandoli:26}), stochastic Navier-Stokes and others.

More recently, some viability results were studied for the porous media
equation by Ciotir, Goreac and Munteanu in \cite{Ciotir/Goreac/Munteanu:23}
and \cite{Ciotir/Goreac/Munteanu:23Asia}. Those cases are partially covered by
the present work, together with other examples. We need to mention that the
present work is using a different, more general approach in a general
Hilbert-Gelfand triple.

To the best of our knowledge, no Nagumo-type characterization has yet been
established for stochastic evolution equations in the general variational
Gelfand triple framework considered here. This setup can accommodate both the
infinite-dimensional geometry and the probabilistic structure of the equation.

This gap is precisely where the present paper is situated. We study viability
properties for stochastic differential equations in a general Hilbert-Gelfand
triple $V\subset H\subset V^{\ast}$, with dynamics governed by
\[
dX_{t}=A(X_{t})\,dt+B(X_{t})\,dW_{t},
\]
where the solution is interpreted in the variational sense and the stochastic
forcing is driven by a cylindrical Wiener process. The constraints are modeled
as closed convex random sets, possibly time dependent, so that the admissible
region itself may evolve with the stochastic environment. Such moving sets
naturally arise in constrained SPDEs, stochastic control, and
state-constrained evolution problems, where the feasible domain is not fixed a
priori but may depend on time, randomness, or external parameters. The main
objective of the paper is to provide a necessary and sufficient
characterization of viability in terms of adapted variational tangent and
contingent sets, thereby extending the geometric viability approach of Aubin
and Da Prato to a broad infinite-dimensional variational setting.

The motivation for such a result is both theoretical and practical. From the
theoretical point of view, the variational approach gives access to a rich
class of nonlinear stochastic dynamics, but it typically yields only
existence, uniqueness, and a priori bounds. By contrast, viability theory is
concerned with local compatibility conditions that determine whether a
stochastic flow can remain inside a closed constraint. Combining the two
creates a substantially more delicate problem, because the geometry of the
constraint must be understood in the ambient Hilbert space while the
stochastic flow is controlled through estimates in the variational space $V$
and its dual. From the practical point of view, viability is indispensable in
applications where admissibility cannot be violated, such as
positivity-preserving models, bounded-state systems, or constrained random
evolution under uncertainty.

\subsection{Personal contributions}

Technically, our analysis is based on a framework of hypotheses that is
natural in variational SPDE theory: hemicontinuity, weak monotonicity,
coercivity, and boundedness ensure the existence and uniqueness of solutions,
while additional assumptions such as local Lipschitz continuity of the drift,
dual positivity, and linear growth of the diffusion are imposed to obtain the
viability characterization. These hypotheses are strong enough to make the
tangent-set machinery effective, yet broad enough to cover relevant nonlinear
models. In particular, the paper points to a number of motivating examples,
including equations with diffusion operators of divergence form, semilinear
equations with Dirichlet or Neumann boundary conditions, a stochastic
Navier-Stokes-type equation, and a porous media model. The variety of examples
demonstrates that the abstract assumptions are not merely formal, but are
compatible with concrete stochastic PDEs arising in applied analysis; see, for
example, Duvaut and Lions \cite{Duvaut/Lions:76} for further details.

The main result establishes an equivalence between the viability of the random
closed set and the inclusion of the pair $(A(x),B(x))$ in the adapted
stochastic variational tangent or contingent set associated with the
constraint. This result can be viewed as a Nagumo-type theorem for variational
SPDEs with random constraints. The paper is organized as follows. After
recalling, in Section \ref{Framework}, the structure of the Gelfand triple and
the assumptions needed for the variational formulation, we introduce the
random constraints and the associated notions of variational tangent and
contingent sets. Before advancing to the main results, we first present, in
Section \ref{Motivating examples}, a collection of motivating examples
illustrating how the abstract theory applies to nonlinear stochastic PDEs and
related models. We then state, in Section \ref{Main results}, the main
viability result and prove, in Section \ref{Proofs}, the equivalence between
viability and the tangent-set characterization using estimates on the solution
process and the boundary behavior of the constraint.

More precisely, we consider a general stochastic differential equation on a
Hilbert space $H$ of the following form
\begin{equation}
dX_{t}=A\left(  X_{t}\right)  dt+B\left(  X_{t}\right)  dW_{t},
\label{GeneralFormEq.}%
\end{equation}
where $W_{t},$ $t\in\left[  0,T\right]  $ is a cylindrical $Q-$Wiener process
on a complete filtered probability space $(\Omega,\mathcal{F},\mathbb{F=}%
\left(  \mathcal{F}_{t}\right)  _{t\geq0},\mathbb{P)}$ with $Q=I$ on another
separable Hilbert space $\left(  U,\left\langle ,\right\rangle _{U}\right)  $.

The existence and uniqueness of solutions to equations of this type were
initially studied under monotonicity conditions by Pardoux (\cite{Pardoux:75})
and comprehensive studies were further developed by Krylov and Rozovskii
\cite{Krylov/Rosowskii:79} and \cite{Rosowskii:90}. To investigate viability
within a variational framework for such a general class of equations, we need
to consider a Hilbert-Gelfand triple, providing also several examples that fit
within this framework.

\section{Framework and spaces setup\label{Framework}}

Let us introduce now, in a rigorous manner, the space framework of our study.
We shall first recall the construction of a generic Gelfand triple of the form
$V\subset H\subset V^{\ast}$. Let $H$ be a separable Hilbert space with the
corresponding scalar product $\left\langle ,\right\rangle _{H}$ and $H^{\ast}$
being its dual. We take $V$ a Banach space such that $V\subset H$ continuously
and densely. It follows that the dual space $V^{\ast}$ verifies that $H^{\ast
}\subset V^{\ast}$ continuously and densely. After identifying $H$ and
$H^{\ast}$ via the Riesz isomorphism we have that%
\[
V\subset H\subset V^{\ast},
\]
continuously and densely, with the embedding positive constant $\mathcal{C}%
_{V,H}~,$ respectively, $\mathcal{C}_{H,V^{\ast}}$. If $_{V^{\ast}%
}\left\langle ,\right\rangle _{V}$ denotes the dualization between $V$ and
$V^{\ast}$, we have that%
\[
_{V^{\ast}}\left\langle z,v\right\rangle _{V}=\left\langle z,v\right\rangle
_{H},\quad\text{for all }z\in H\text{ and }v\in V\text{.}%
\]
Kuratowski's theorem states that $V\in\mathcal{B}\left(  H\right)  ,$
$H\in\mathcal{B}\left(  V^{\ast}\right)  $ and $\mathcal{B}\left(  V\right)
=\mathcal{B}\left(  H\right)  \cap V,$ $\mathcal{B}\left(  H\right)
=\mathcal{B}\left(  V^{\ast}\right)  \cap H.$

\begin{remark}
\label{remark 1 spaces projection} In our case we need also to take $V$ and
$V^{\ast}$ as being Hilbert spaces\footnote{The need to work under a
Hilbert-Gelfand triple setup still permits to provide consistent models which
can be covered by this assumption. More details will be presented in
Subsection \ref{Motivating examples}.} in order to get a duality mapping
$J:V\rightarrow V^{\ast}$ defined, for every $x\in V,$ by
\[
J\left(  x\right)  =\left\{  x^{\ast}\in V^{\ast}~;~_{V^{\ast}}\left\langle
x^{\ast},x\right\rangle _{V}=\left\Vert x\right\Vert _{V}^{2}=\left\Vert
x^{\ast}\right\Vert _{V^{\ast}}^{2}\right\}  .
\]
We define $J_{H}:D\left(  J_{H}\right)  \subset V\rightarrow H,$\ where%
\[
\left\{
\begin{array}
[c]{rcl}%
D\left(  J_{H}\right)  & = & \left\{  y\in V;~J\left(  y\right)  \in
H\right\}  ,\smallskip\\
J_{H}\left(  y\right)  & = & J\left(  y\right)  ,\quad\forall y\in D\left(
J_{H}\right)  .
\end{array}
\right.
\]
One can easily check that $J_{H}$ is self-adjoint and positive mapping defined
in $H$. Indeed, for all $z,y\in D\left(  J_{H}\right)  $ we have that
\[
\left\langle J_{H}\left(  z\right)  ,y\right\rangle _{H}=_{V^{\ast}%
}\left\langle J\left(  z\right)  ,y\right\rangle _{V}=\left\langle
z,y\right\rangle _{V}=_{V}\left\langle z,J\left(  y\right)  \right\rangle
_{V^{\ast}}=\left\langle z,J_{H}\left(  y\right)  \right\rangle _{H}%
\]
and
\[
\left\langle J_{H}\left(  y\right)  ,y\right\rangle _{H}=_{V^{\ast}%
}\left\langle J\left(  y\right)  ,y\right\rangle _{V}=\left\Vert y\right\Vert
_{V}^{2}.
\]
This assures the existence of a square root operator $J_{H}^{1/2},$ such that
$J_{H}=J_{H}^{1/2}\circ J_{H}^{1/2}.$

In our case, we shall use it in order to prove that the solution is bounded in
$V.$ More precisely, we have the following equalities:%
\begin{equation}
\left\Vert x\right\Vert _{V}^{2}=\left\langle x,x\right\rangle _{V}=_{V^{\ast
}}\left\langle J_{H}\left(  x\right)  ,x\right\rangle _{V}=\left\langle
J_{H}\left(  x\right)  ,x\right\rangle _{H}=\left\Vert J_{H}^{1/2}\left(
x\right)  \right\Vert _{H}~, \label{norm equality}%
\end{equation}
for all $x\in D\left(  J_{H}\right)  $. Since the restricted projection
$J_{H}$ defined above\ is $m-$ accretive in $H\times H$, we can take its
Yosida approximation $J_{H}^{\varepsilon}$, i.e.%
\[
J_{H}^{\varepsilon}\left(  y\right)  :=\frac{1}{\varepsilon}\left(  y-\left(
1+\varepsilon J_{H}\right)  ^{-1}\left(  y\right)  \right)  \in J_{H}\left(
\left(  1+\varepsilon J_{H}\right)  ^{-1}\left(  y\right)  \right)  ,
\]
for all $y\in H.$ By the classical theory, we have that $J_{H}^{\varepsilon}$
is also self-adjoint and positively defined and, therefore, we have that
$\left(  J_{H}^{\varepsilon}\right)  ^{1/2}$ is the square root operator of
$J_{H}^{\varepsilon}$, (i.e. $J_{H}^{\varepsilon}=\left(  J_{H}^{\varepsilon
}\right)  ^{1/2}\circ\left(  J_{H}^{\varepsilon}\right)  ^{1/2}$).
\end{remark}

Referring to the drift and diffusion operators driving
Eq.(\ref{GeneralFormEq.}),%
\[
A:V\times\Omega\rightarrow V^{\ast}\quad\text{and}\quad B:V\times
\Omega\rightarrow L_{2}\left(  U,H\right)  ~\footnote{$L_{2}\left(
U,H\right)  $ denotes the space of Hilbert-Schmidt operators from $U$ to $H$%
.},
\]
which are assumed to be $\left(  \mathcal{B}\left(  V\right)  \otimes
\mathbb{F},\mathcal{B}\left(  V^{\ast}\right)  \right)  -,$ respectively
$\left(  \mathcal{B}\left(  V\right)  \otimes\mathbb{F},\mathcal{B}\left(
H\right)  \right)  -$ adapted.

We shall first recall the classical hypotheses which are necessary to obtain
the existence and uniqueness of the solution in infinite dimensional spaces,
also known as the \textit{variational approach}. For more details, the
interested reader is invited to consult Pr\'{e}vot, Rockner \cite[Chapter
4]{Prevot/Rockner:07}.$\smallskip$

\noindent\textbf{Hypotheses I. }The following assumptions should be understood
to take place $\mathbb{P}$-$a.e.\omega\in\Omega:$

\begin{itemize}
\item[$\left(  H_{1}\right)  $] (\textit{Hemicontinuity}) For all $u,v,x\in V$
the map
\[
r\mapsto_{V^{\ast}}\left\langle A\left(  u+rv\right)  ,x\right\rangle
_{V}\text{,}\quad\text{for }r\in\mathbb{R}\text{,}%
\]
is continuous.

\item[$\left(  H_{2}\right)  $] (\textit{Weak monotonicity}) There exists
$c\in\mathbb{R}$ such that, for all $u,v\in V$,%
\[
_{V^{\ast}}\left\langle A\left(  u\right)  -A\left(  v\right)
,u-v\right\rangle _{V}+\left\Vert B\left(  u\right)  -B\left(  v\right)
\right\Vert _{L_{2}\left(  U,H\right)  }^{2}\leq C_{A,B}\left\Vert
u-v\right\Vert _{H}^{2}.
\]

\item[$\left(  H_{3}\right)  $] (\textit{Coercivity}) There exist $\alpha
\in(1,\infty),$ $C_{1}\in\mathbb{R}$, $C_{2}\in(0,\infty)$ such that, for all
$v\in V,$ we have%
\[
_{V^{\ast}}\left\langle A\left(  v\right)  ,v\right\rangle _{V}+\left\Vert
B\left(  v\right)  \right\Vert _{L_{2}\left(  U,H\right)  }^{2}\leq
C_{1}\left\Vert v\right\Vert _{H}^{2}-C_{2}\left\Vert v\right\Vert
_{V}^{\alpha}~.
\]

\item[$\left(  H_{4}\right)  $] (\textit{Boundedness}) There exist $c_{3}%
\in(0,\infty)$ such that, for all $v\in V,$ we have%
\[
\left\Vert A\left(  v\right)  \right\Vert _{V^{\ast}}\leq C_{3}\left\Vert
v\right\Vert _{V}^{\alpha-1}.
\]

\end{itemize}

The second set of hypotheses are necessary only for the viability
result.$\smallskip$

\noindent\textbf{Hypotheses II.}

\begin{itemize}
\item[$\left(  H_{5}\right)  $] (\textit{Locally} \textit{Lipschitz}) For
every $R>0,$ there exists a positive constant $L_{A,R}$ such that, for every
$u,v\in V,$ with $\left\Vert u\right\Vert _{V}\leq R,\left\Vert v\right\Vert
_{V}\leq R,$%
\[
\left\Vert A\left(  u\right)  -A\left(  v\right)  \right\Vert _{V^{\ast}}\leq
L_{A,R}\left\Vert u-v\right\Vert _{V}~.
\]

\item[$\left(  H_{6}\right)  $] (\textit{Dual positivity}) For $J:V\rightarrow
V^{\ast}$ the duality mapping we have that
\[
\left\langle -A\left(  u\right)  ,J_{H}^{\varepsilon}\left(  u\right)
\right\rangle _{H}\geq0~,
\]
for all $u\in D\left(  J_{H}\right)  \subset V,$ such that $A\left(  u\right)
\in H$.\medskip

\item[$\left(  H_{7}\right)  $] \textit{(Linear noise) }The operator $B$ is
linear and there exists a positive constant $M_{B}$ such that%
\[
\left\Vert B\left(  v\right)  \right\Vert _{L_{2}\left(  U,H\right)  }^{2}\leq
M_{B}\left\Vert v\right\Vert _{H}^{2}~,
\]
for all $v\in H$.
\end{itemize}

Under \textbf{Hypotheses I}, for each initial condition $X\left(  0\right)
\in L^{2}\left(  \Omega;H\right)  $ we can study equation
(\ref{GeneralFormEq.}) in the sense of the variational approach (see Pardoux
\cite{Pardoux:75}). More precisely, the equation has a unique strong solution
by applying Theorem 1.1, page 83. \textbf{Hypotheses II} are necessary for the
viability result.

\begin{theorem}
[Pardoux \cite{Pardoux:75}, Theorem 1.1]Under Hypotheses I $(\left(
H_{1}\right)  $ to $\left(  H_{4}\right)  )$ admits a unique solution%
\[
X\in L^{p}\left(  \Omega\times\left[  0,T\right]  ;V\right)  \cap L^{2}\left(
\Omega;C\left(  0,T;H\right)  \right)  .
\]

\end{theorem}

\subsection{Some motivating examples\label{Motivating examples}}

The above assumptions permit various and quite general classes of problems to
be covered by the present study. The need to work under a Hilbert-Gelfand
triple setup restriction does not exclude important type of space environments
and models, as one can see below. For the clarity of the presentation, when
needed, we prove in Appendix \ref{Appendix} that the coefficients of the
described problems are in concordance with the hypothesis we assumed for our
research. We present now some examples/situations which can be covered by the
general form we considered, the main focus being on the structure of $A.$

\begin{example}
[\textbf{Equations with divergence-type drift}]$~\smallskip$

We consider the Hilbert-Gelfand triple $H_{0}^{1}\left(  \mathcal{O}\right)
\subset L^{2}\left(  \mathcal{O}\right)  \subset H^{-1}\left(  \mathcal{O}%
\right)  ,$ which is suitable for the equation of the type%
\begin{equation}
\left\{
\begin{array}
[c]{ll}%
dX_{t}=\operatorname{div}(a\left(  \nabla X_{t}\right)  )dt+B\left(
X_{t}\right)  dW_{t} & ,\text{ }in\text{ }\left(  0,\infty\right)
\times\mathcal{O},\medskip\\
X_{t}=0 & ,\text{ }on\text{ }\left(  0,\infty\right)  \times\partial
\mathcal{O},\medskip\\
X_{0}=x & ,\text{ }in\text{ }\mathcal{O},
\end{array}
\right.  \label{ex of eq. with divergence-type op}%
\end{equation}
where $\mathcal{O}$\ is a non-empty bounded open domain in $%
%TCIMACRO{\U{211d} }%
%BeginExpansion
\mathbb{R}
%EndExpansion
^{d}$, $d\geq0$, with smooth boundary.

One can easily see that this equation can be rewritten in the framework of our
paper by taking the operator%
\[
\left\{
\begin{array}
[c]{l}%
A:H_{0}^{1}\left(  \mathcal{O}\right)  \rightarrow H^{-1}\left(
\mathcal{O}\right)  ~,\medskip\\
A\left(  u\right)  :=\operatorname{div}(a\left(  \nabla u\right)
,\quad\forall u\in H_{0}^{1}\left(  \mathcal{O}\right)  ,
\end{array}
\right.
\]
which means, in variational form,%
\[
\left(  A\left(  u\right)  ,v\right)  _{L^{2}\left(  \mathcal{O}\right)
}=-\int\limits_{\mathcal{O}}a\left(  \nabla u\right)  \nabla vd\xi
,\quad\forall u,v\in H_{0}^{1}\left(  \mathcal{O}\right)  .
\]

In order to situate Eq.(\ref{ex of eq. with divergence-type op}) in our
constructed research environment, we assume that the operator $a:\mathbb{R}%
^{d}\rightarrow\mathbb{R}^{d}$ is a map of the type $a=\nabla j,$ where
$j:\mathbb{R}^{d}\rightarrow\mathbb{R}$ is a proper convex, differentiable
function. Moreover,

\begin{itemize}
\item[$a)$] there exist $C_{1}>0$ and $C_{2}\geq0$ such that%
\[
j\left(  r\right)  \geq C_{1}\left\vert r\right\vert _{\mathbb{R}^{d}}%
^{2}-C_{2},\quad\forall r\in\mathbb{R}^{d}~;
\]

\item[$b)$] there exists $C_{3}>0$ such that%
\[
\left\vert a\left(  r\right)  \right\vert _{\mathbb{R}^{d}}\leq C_{3}\left(
\left\vert r\right\vert +1\right)  ,\quad\forall r\in\mathbb{R}^{d}~;
\]

\item[$c)$] there exists $C_{4}>0$ such that%
\[
\left\langle a\left(  r\right)  -a\left(  s\right)  ,r-s\right\rangle
_{\mathbb{R}^{d}}\geq C_{4}\left\vert r-s\right\vert _{\mathbb{R}^{d}}%
^{2}~,\quad\forall r,s\in\mathbb{R}^{d}~.
\]

\end{itemize}

For more informations concerning the need of assumption $\left(  H_{6}\right)
,$ the reader is invited to see the detailed explanations provided by Barbu,
Rockner \cite{Barbu/Rockner:13}. Even, at first impression, it seems a bit
artificial, it is well compatibles with a lot of important models which can be
address by our study.$\smallskip$

Two particular forms of the operator $A$ can be enumerated below.$\smallskip$

\noindent$\left(  i\right)  $ operator for a model of electrostatic nonlinear
conductivity:%
\[
A\left(  u\right)  :=\operatorname{div}\left(  \sigma\left(  u\right)  \nabla
u\right)  ,
\]
where we consider the strictly positive Lipschitz function $\sigma$, given by
$\sigma\left(  u\right)  =1+1/\left(  1+u^{2}\right)  .\smallskip$

\noindent$\left(  ii\right)  $ operator used in image processing by non-linear
filtration:%
\[
A\left(  u\right)  :=\operatorname{div}\left(  g\left(  \left\vert \nabla
u\right\vert \right)  \nabla u\right)  ,
\]
where $g\left(  s\right)  =1/\left(  1+s^{2}\right)  .$ This choice is a
regularized version of the Perona-Malik diffusion model, with applications to
image denoising and edge-preserving smoothing. However, the classical form is
not Lipschitz, but, by regularization, with $g_{\varepsilon}\left(  s\right)
=1/\left(  1+s^{2}+\varepsilon\right)  ,$ $\varepsilon>0,$ it becomes well
defined, locally Lipschitz and positively energetic.\medskip

For the diffusion part of (\ref{ex of eq. with divergence-type op}), the
operator $B:H_{0}^{1}\left(  \mathcal{O}\right)  \rightarrow L_{2}\left(
L^{2}\left(  \mathcal{O}\right)  ;L^{2}\left(  \mathcal{O}\right)  \right)  $
is assumed to be linear from $H_{0}^{1}\left(  \mathcal{O}\right)  $ to the
space of Hilbert-Schmidt operators $L_{2}\left(  L^{2}\left(  \mathcal{O}%
\right)  ;L^{2}\left(  \mathcal{O}\right)  \right)  $.
\end{example}

\begin{example}
[\textbf{Semi-linear Dirichlet equations}]%
\label{Example semi-linear Dirichlet eq.}$~\smallskip$

As above, we consider the same Hilbert-Gelfand triple $H_{0}^{1}\left(
\mathcal{O}\right)  \subset L^{2}\left(  \mathcal{O}\right)  \subset
H^{-1}\left(  \mathcal{O}\right)  ,$ which is also suitable for the equation
of the type%
\begin{equation}
\left\{
\begin{array}
[c]{ll}%
dX_{t}=\Delta X_{t}dt-\beta\left(  X_{t}\right)  dt+B\left(  X_{t}\right)
dW_{t} & ,\text{ }in\text{ }\left(  0,\infty\right)  \times\mathcal{O}%
,\medskip\\
X_{t}=0 & ,\text{ }on\text{ }\left(  0,\infty\right)  \times\partial
\mathcal{O},\medskip\\
X_{0}=x & ,\text{ }in\text{ }\mathcal{O},
\end{array}
\right.  \label{Semilinear Dirichlet eq.}%
\end{equation}
where $\mathcal{O}$\ is a non-empty bounded open domain in $\mathbb{R}^{d}$,
$d\geq0$, with smooth boundary. This equation can be rewritten in the
framework of our paper by taking the operator%
\[
A:H_{0}^{1}\left(  \mathcal{O}\right)  \rightarrow H^{-1}\left(
\mathcal{O}\right)
\]
which means in variational form,%
\begin{equation}
\left(  A\left(  u\right)  ,v\right)  _{L^{2}(\mathcal{O})}=-\int
\limits_{\mathcal{O}}\nabla u\nabla vd\xi-\int\limits_{\mathcal{O}}%
\beta\left(  u\right)  vd\xi,\quad\forall u,v\in H_{0}^{1}\left(
\mathcal{O}\right)  . \label{representation of operator A}%
\end{equation}
In order to be able to include Eq.(\ref{Semilinear Dirichlet eq.}) under the
hypothesis of our general result, let us assume that the operator $\beta:$ $%
%TCIMACRO{\U{211d} }%
%BeginExpansion
\mathbb{R}
%EndExpansion
\rightarrow$ $%
%TCIMACRO{\U{211d} }%
%BeginExpansion
\mathbb{R}
%EndExpansion
$ is monotonically increasing, $\beta\left(  0\right)  =0$ and satisfies, for
some positive constants $C_{1},...,C_{5}$ and every $r,s\in\mathbb{R},$

\begin{itemize}
\item[$a)$] $\left\vert \beta\left(  r\right)  \right\vert \leq C_{1}%
\left\vert r\right\vert +C_{2}~;$

\item[$b)$] $\beta\left(  r\right)  r\geq C_{3}r^{2}-C_{4},$~$;$

\item[$c)$] $(\beta\left(  r\right)  -\beta\left(  s\right)  )(r-s)\geq
C_{5}(r-s)^{2}.$
\end{itemize}

Some particular forms of the operator $A$ can be enumerated below:$\smallskip$

\noindent$\left(  i\right)  $ the \textit{clipping} operator%
\[
A\left(  u\right)  :=\Delta u-\operatorname*{clip}(u),
\]
where%
\[
\operatorname*{clip}(u)=\max\left(  -M,\min\left(  u,M\right)  \right)  ,\quad
M>0\text{ fixed,}%
\]
is Lipschitz and of monotone type. This operator can be used in convex
optimization problems, constrained control problems and deep learning
modelling.$\smallskip$

\noindent$\left(  ii\right)  $ operator with Lipschitz reaction term:%
\[
A\left(  u\right)  :=\Delta u-f\left(  u\right)  ,
\]
where $f$ is Lipschitz or increasing, of the type $f(u)=\lambda u$ or
$f(u)=\alpha\tanh\left(  u\right)  ,$ $\alpha>0,$ which corresponds to a
diffusion with saturated reaction. Since $\tanh$ is monotone and globally
Lipschitz, it implies that%
\[
\left\langle A\left(  u\right)  ,u\right\rangle =\int\left\vert \nabla
u\right\vert ^{2}+\alpha\int\tanh\left(  u\right)  u\geq0.
\]

\noindent$\left(  iii\right)  $ nonlinear saturated operator:%
\[
A\left(  u\right)  :=\Delta u-\frac{u}{1+\left\vert u\right\vert };
\]

\noindent$\left(  iv\right)  $ Nemytskii type operator $A\left(  u\right)
:=\Delta u-g\left(  x,u\right)  ,$ where $g\left(  x,\cdot\right)  $ is
Lipschitz, uniformly in $x$ and monotone with respect to $u.$\bigskip

For the diffusion of the Eq.(\ref{Semilinear Dirichlet eq.}), the operator
$B:H_{0}^{1}\left(  \mathcal{O}\right)  \rightarrow L_{2}\left(  L^{2}\left(
\mathcal{O}\right)  ;L^{2}\left(  \mathcal{O}\right)  \right)  $ is assumed to
be linear from $H_{0}^{1}\left(  \mathcal{O}\right)  $ to the space of
Hilbert-Schmidt operators $L_{2}\left(  L^{2}\left(  \mathcal{O}\right)
;L^{2}\left(  \mathcal{O}\right)  \right)  $.
\end{example}

\begin{example}
[\textbf{Semi-linear Neumann equations}]$~\smallskip$

Let us consider now the Hilbert-Gelfand triple $H^{1}\left(  \mathcal{O}%
\right)  \subset L^{2}\left(  \mathcal{O}\right)  \subset\left(  H^{1}\left(
\mathcal{O}\right)  \right)  ^{\ast}$ and the semi-linear Neumann equation:%
\begin{equation}
\left\{
\begin{array}
[c]{ll}%
dX_{t}=\Delta X_{t}dt+B\left(  X_{t}\right)  dW_{t} & ,\text{ }in\text{
}\left(  0,\infty\right)  \times\mathcal{O},\medskip\\
\dfrac{\partial X_{t}}{\partial\nu}+\beta\left(  X_{t}\right)  =0 & ,\text{
}on\text{ }\left(  0,\infty\right)  \times\partial\mathcal{O},\medskip\\
X_{0}=x & ,\text{ }in\text{ }\mathcal{O},
\end{array}
\right.  \label{semin-linear Neumann eq}%
\end{equation}
where $\mathcal{O}$\ is a non-empty bounded open domain in $%
%TCIMACRO{\U{211d} }%
%BeginExpansion
\mathbb{R}
%EndExpansion
^{d}$, $d\geq0$, with smooth boundary. To comply with the general form
considered for the viability result, we consider the operator%
\[
A:H^{1}\left(  \mathcal{O}\right)  \rightarrow\left(  H^{1}\left(
\mathcal{O}\right)  \right)  ^{\ast}%
\]
which means, in variational form,%
\[
\left(  A\left(  u\right)  ,v\right)  _{L^{2}(\mathcal{O})}=-\int
\limits_{\mathcal{O}}\nabla u\nabla vd\xi-\int\limits_{\partial\mathcal{O}%
}\operatorname{Tr}\beta\left(  u\right)  \operatorname{Tr}vd\xi,\quad\forall
u,v\in H^{1}\left(  \mathcal{O}\right)  ,
\]
where we denoted by $\operatorname{Tr}$ the trace operator. Concerning
$\beta,$ it verifies exactly the same assumptions as in Example
\ref{Example semi-linear Dirichlet eq.}, above. We present below two
particular forms for (\ref{semin-linear Neumann eq}). For some given barriers
$a<b,$ we have:$\smallskip$

\noindent$\left(  i\right)  $ black body radiation (Stefan-Boltzmann radiation
law)%
\[
\beta(r):=ar^{4}+b,\quad\text{for }a\leq r\leq b~;
\]

\noindent$\left(  ii\right)  $ models of motion through permeable membranes%
\[
\beta(r):=\left\{
\begin{array}
[c]{rl}%
\alpha_{1}(r-a) & \text{,}\quad\text{if }r<a,\smallskip\\
0 & \text{,}\quad\text{if }a\leq r\leq b,\smallskip\\
\alpha_{2}(r-b) & \text{,}\quad\text{if }r>b,
\end{array}
\right.
\]

\bigskip For the diffusion part of (\ref{semin-linear Neumann eq}), the
operator $B:H^{1}\left(  \mathcal{O}\right)  \rightarrow L_{2}\left(
L^{2}\left(  \mathcal{O}\right)  ;L^{2}\left(  \mathcal{O}\right)  \right)  $
is assumed to be linear from $H^{1}\left(  \mathcal{O}\right)  $ to the space
of Hilbert-Schmidt operators $L_{2}\left(  L^{2}\left(  \mathcal{O}\right)
;L^{2}\left(  \mathcal{O}\right)  \right)  $.
\end{example}

\begin{example}
[\textbf{A Stochastic Navier-Stokes-type Equation}]\label{SNSE example}
Consider the Hilbert-Gelfand triple $V\subset H\subset V^{\ast},$ where%
\[
\left\{
\begin{array}
[c]{ccccl}%
V & := & H_{0,\sigma}^{1}\left(  D\right)  & = & \left\{  u\in L^{2}\left(
D;\mathbb{R}^{d}\right)  :\nabla\cdot u=0\text{ in }D,\text{ }u\cdot
n|_{\partial D}=0\right\}  ,\medskip\\
H & := & L_{\sigma}^{2}\left(  D\right)  & = & \left\{  u\in H_{0}^{1}\left(
D;\mathbb{R}^{d}\right)  :\nabla\cdot u=0\right\}  ,\medskip\\
V^{\ast} & := & H_{\sigma}^{-1}\left(  D\right)  , &  &
\end{array}
\right.
\]
where $D\subset\mathbb{R}^{2}$ is a bounded smooth domain and $\sigma$ denotes
divergence-free vector fields. Condition $\nabla\cdot u=0$ means that the
fluids' speed is incompressible, that is the fluid does not change, locally,
its volume.

In order to construct some Stochastic Navier-Stokes-like Equations (SNSEs),
which can be approached by our study, let us first remark that the dual
positivity condition asked by $\left(  H_{6}\right)  $ is the most delicate
one, since, generally, the Stokes operator is a dissipative one. Also, the
hypothesis $\left(  H_{5}\right)  $ requires only a local Lipschitz behavior
for $A.$ However, for classical SNSEs, if we want to verify the hypothesis
$\left(  H_{6}\right)  ,$ we must absorb the action of $\left(  u\cdot
\nabla\right)  u$ by the rest of the terms. This is the main reason why we
should consider a tamed SNSE-type as a relevant suitable example. In addition,
this model is driven by a global Lipschitz drift: given a positive cinematic
viscosity coefficient $\nu,$ let us consider the dynamic%
\[
dX_{t}=A_{\mathrm{reg}}\left(  X_{t}\right)  dt+B\left(  X_{t}\right)
dW_{t}~,\quad X_{0}=x\in H,
\]
where the regularized coefficient is given by%
\begin{equation}
A_{\mathrm{reg}}\left(  u\right)  :=-\nu Au-\chi_{R}\left(  \left\Vert
u\right\Vert _{V}\right)  \mathcal{B}\left(  u,u\right)  -g\left(  \left\Vert
u\right\Vert _{H}^{2}\right)  u~. \label{tamed SNSE A term}%
\end{equation}
Here, the Stokes operator is, for every $u,v\in V,$%
\[
_{V^{\ast}}\left\langle Au,v\right\rangle _{V}=\left\langle \nabla u,\nabla
v\right\rangle _{H},\quad i.e.\quad A:=-P\Delta,
\]
with $P$ being the Leray projection and $\mathcal{B}$ is the convective
bilinear operator of the tamed SNSE%
\[
\mathcal{B}\left(  u,v\right)  :=P((u\cdot\nabla)v),
\]
which, used on its diagonal, produces the main nonlinearity of the equation.
The regularization cutoff function $\chi_{R}\in C^{1}\left(  [0,\infty
);[0,1]\right)  $ verifies that $\chi_{R}(r)=1,$ for $r\leq R,$ and $\chi
_{R}(r)=0,$ for $r\geq R+1,$ while $g:[0,\infty)\rightarrow\lbrack0,\infty)$
is bounded, globally Lipschitz, and nonnegative.$\smallskip$

The diffusion is of linear type, $B :=\Sigma u,$ with bounded $\Sigma\in
L(H,L_{2}(U,H)).$

The interested reader can consult in Appendix \ref{Appendix} the details which
proves that the working assumptions are verified for the above problem.
\end{example}

\begin{example}
[\textbf{A porous media model}]Consider, the Hilbert-Gelfand triple $V\subset
H\subset V^{\ast},$ with continuous and dense embeddings, where%
\[
V=L^{2}(\mathcal{O}),\quad H=H^{-1}(\mathcal{O}),\quad V^{\ast}=\left(
L^{2}\left(  \mathcal{O}\right)  \right)  ^{\ast},
\]
with $\mathcal{O}$ being a nonempty bounded domain from $\mathbb{R}^{d};$
$\left(  L^{2}\left(  \mathcal{O}\right)  \right)  ^{\ast}$ should be
understood as the dual of $L^{2}(\mathcal{O}),$ but in the topology of
$H^{-1}(\mathcal{O}),$ which is not the same as $L^{2}\left(  \mathcal{O}%
\right)  ,$ as one can see in Pr\'{e}vot, R\"{o}ckner \cite[Remark
4.1.14.]{Prevot/Rockner:07}. More precisely,%
\[
\left(  L^{2}\left(  \mathcal{O}\right)  \right)  ^{\ast}=\Delta\left(
L^{2}\left(  \mathcal{O}\right)  \right)  \neq L^{2}\left(  \mathcal{O}%
\right)  .
\]
We construct now a porous-media model:%
\[
dX_{t}=A\left(  X_{t}\right)  dt+B\left(  X_{t}\right)  dW_{t}~,\quad
X_{0}=x\in H\cap L^{1}\left(  \mathcal{O}\right)  ,
\]
where%
\[
A\left(  u\right)  :=\Delta\beta\left(  u\right)  ,
\]
for $\beta:\mathbb{R}\rightarrow\mathbb{R}$ chosen as in Example
\ref{Example semi-linear Dirichlet eq.}.

One can easily see that the usual assumptions at verified in the classical
way. Concerning the assumption $\left(  H_{6}\right)  $ we have
\[
(-A(u),J_{H}(v))_{H^{-1}(\mathcal{O})}=(-\Delta\beta(u),-\Delta(v))_{H^{-1}%
(\mathcal{O})}=(-\Delta\beta(u),v))_{L^{2}(\mathcal{O})}=(\nabla
\beta(u),\nabla v))_{L^{2}(\mathcal{O})}\geq0.
\]

Also, $B\left(  u\right)  :=\Sigma u,$ with bounded $\Sigma\in L(H,L_{2}%
(U,H)).$ It is easy to verify that hypotheses $\left(  H_{1}\right)  $ to
$\left(  H_{7}\right)  $ are verified.
\end{example}

\section{Main viability results\label{Main results}}

We first define the constraints as nonempty closed subsets $K_{\omega}\subset
V$, where the set valued map
\[
K:\omega\in\Omega\rightrightarrows K_{\omega}%
\]
is assumed to be $\mathcal{F}_{0}$\ - measurable, which means that $K$ is a
random set-valued variable. Let us introduce the subset $\mathcal{K}$ given by%
\[
\mathcal{K}:=\left\{  u\in L^{2}\left(  \Omega;V\right)  :\text{ for almost
all }\omega\in\Omega\text{, }u\left(  \omega\right)  \in K_{\omega}\right\}
.
\]
We define the \textit{stochastic viability property} of $K$ with respect to
the pair of operators $\left(  A,B\right)  $ driven equation
(\ref{GeneralFormEq.}) as follows:

\begin{definition}
We say that $K$ is \textbf{viable}\ with respect to the pair of operators
$\left(  A,B\right)  $ if, for every $t\geq0$ and any $\mathcal{F}_{t}-$random
variable $x\in K,$ there exists a stochastic process $X:=(X_{s}^{t,x})_{s\geq
t}$ which is a solution of the SDE (\ref{GeneralFormEq.}) starting from $x$,
such that $X$ is viable in $K$. More precisely,
\[
\forall~0\leq t\leq s\leq T,\quad X_{s}^{t,x}\left(  \omega\right)  \in
K\left(  \omega\right)  ,\quad\mathbb{P-}a.e\text{ ~}\omega\in\Omega.
\]

\end{definition}

In order to give a Nagumo type characterization of the viability property, we
need first to adapt the concepts of \textit{variational tangent set
}and\textit{ variational contingent set} to the situation of stochastic
partial differential equation in variational formulation.

\begin{definition}
\label{defS}Let us consider an arbitrarily fixed $t\geq0$ and $x\in K$ an
$\mathcal{F}_{t}$ - random variable from $V$. We define the $\mathcal{F}_{t}%
-$measurable \textbf{variational tangent set} $\mathcal{S}_{K}\left(
t,x\right)  $ to $K$ at $x$ as being the set of $\mathcal{F}_{t}$ - random
variable pairs $\left(  \gamma,v\right)  $ from $V^{\ast}\times V^{\ast}$
satisfying the following properties: there exist two adapted continuous
stochastic processes $\alpha$ and $\beta$ with values in $V^{\ast}$ and
respectively in $L_{2}\left(  U;H\right)  $ satisfying, for some fixed
$0<\delta<\frac{1}{2}$%
\[
\lim_{h\rightarrow0}\mathbb{E}\left(  \sup_{r\in\left[  t,t+h\right]
}\left\Vert \alpha\left(  r\right)  \right\Vert _{V^{\ast}}^{2}\right)  =0
\]%
\[
\lim_{h\rightarrow0}\mathbb{E}\left(  \sup_{r\in\left[  t,t+h\right]
}\left\Vert \frac{1}{h^{\delta}}\beta\left(  r\right)  \right\Vert
_{L_{2}\left(  U;H\right)  }^{2}\right)  =0
\]
and%
\[
x+\int_{t}^{t+h}\left(  \gamma+\alpha(s)\right)  ds+\int_{t}^{t+h}\left(
v+\beta\left(  s\right)  \right)  dW_{s}\in\mathcal{K}.
\]
for $h$ small enough.
\end{definition}

\begin{definition}
\label{defT}Let us consider an arbitrarily fixed $t\geq0$ and $x\in K$ an
$\mathcal{F}_{t}$ - random variable from $V$. We define the $\mathcal{F}_{t}%
-$measurable \textbf{variational contingent set} $\mathcal{T}_{K}\left(
t,x\right)  $ to $K$ at $x$ as being the set of $\mathcal{F}_{t}$ - random
variable pairs $\left(  \gamma,v\right)  $ from $V^{\ast}\times V^{\ast}$
satisfying the following properties: there exist a sequence $\left(
h_{n}\right)  _{n\in\mathbb{N}}\subset\mathbb{R}_{+}$ and the corresponding
sequences of $\mathcal{F}_{t+h_{n}}$ - random variables $\left(  \alpha
_{n}\right)  _{n\in\mathbb{N}}$ and $\left(  \beta_{n}\right)  _{n\in
\mathbb{N}},$ such that $h_{n}\rightarrow0$ as $n\rightarrow+\infty$ and%
\begin{equation}
\left\{
\begin{array}
[c]{cl}%
i) & \mathbb{E}\left(  \left\Vert \alpha_{n}\right\Vert _{V^{\ast}}%
^{2}\right)  \rightarrow0\smallskip\\
ii) & \mathbb{E}\left(  \left\Vert \beta_{n}\right\Vert _{H}^{2}\right)
\rightarrow0\text{ }\smallskip\\
iii) & \mathbb{E}\left(  \beta_{n}\right)  =0\smallskip\\
iv) & \beta_{n}\text{ independent of }\mathcal{F}_{t},
\end{array}
\right.  \label{ConditiiAlfaBeta}%
\end{equation}
and satisfying, for every $n\in\mathbb{N},$ the condition%
\[
x+\gamma h_{n}+v\left(  W_{t+h_{n}}-W_{t}\right)  +h_{n}\alpha_{n}%
+h_{n}^{\frac{1}{2}+\delta}\beta_{n}\in\mathcal{K}.
\]

\end{definition}

We can give now the main result of this paper.

\begin{theorem}
\label{main result}Let $K$ be a nonempty closed subsets of $V$. We take $A$
and $B$ the operators from equation (\ref{GeneralFormEq.}) and we assume all
the hypotheses above satisfied. Then the following conditions are equivalent.

\begin{itemize}
\item[$\mathit{i)}$] The subset $K$ is viable with respect to the pair of
operators $\left(  A,B\right)  .$

\item[$\mathit{ii)}$] For every $t\in\left[  0,T\right]  $ and for every
$\mathcal{F}_{t}$ - random variable $x\in L^{2}\left(  \Omega;V\right)  $ from
$K$ we have that%
\[
\left(  A\left(  x\right)  ,B\left(  x\right)  \right)  \in\mathcal{S}%
_{K}\left(  t,x\right)  .
\]

\item[$\mathit{iii)}$] For every $t\in\left[  0,T\right]  $ and for every
$\mathcal{F}_{t}$ - random variable $x\in L^{2}\left(  \Omega;V\right)  $ from
$K$ we have that%
\[
\left(  A\left(  x\right)  ,B\left(  x\right)  \right)  \in\mathcal{T}%
_{K}\left(  t,x\right)  .
\]

\end{itemize}
\end{theorem}

\section{Proofs of the main results\label{Proofs}}

In order to prove the main result we need first to show the following
important estimates from the key Lemma \ref{estim1}. In what follows, we
denote by $C$ a generic positive constant, which can vary from one line to
another if there isn't necessary to provide its explicit dependence on some
other constants involved. Whenever its exact form is mandatory, we will
precisely mention it.

\begin{lemma}
\label{estim1}Under the assumptions from Section \ref{Framework}, we have the
following estimates for the solution of equation (\ref{GeneralFormEq.}):%
\begin{equation}
\mathbb{E}\left(  \underset{r\in\left[  0,T\right]  }{\sup}\left\Vert
X_{r}\right\Vert _{V}^{2}\right)  \leq C, \label{marg-sol}%
\end{equation}%
\begin{equation}
\mathbb{E}\left(  \underset{r\in\left[  t,t+h\right]  }{\sup}\left\Vert
X_{r}-X_{t}\right\Vert _{V}^{2}\right)  \leq Ch,\quad h>0,\text{ fixed,}
\label{marg-dif}%
\end{equation}
where $C$ is a $h-$independent positive constant.
\end{lemma}

\begin{proof}
[Proof of Lemma \ref{estim1}]In order to get the estimates, let us remind that
the restricted projection $J_{H}$ defined in Remark
\ref{remark 1 spaces projection}\ is $m-$ accretive in $H\times H$ and
consider its Yosida approximation $J_{H}^{\varepsilon}$, together with its
square root operator, $\left(  J_{H}^{\varepsilon}\right)  ^{1/2}$.

For the first estimate, we need to apply It\^{o}'s formula to the process%
\[
X_{t}^{0,x}=x+\int_{0}^{t}A(X_{r}^{0,x})dr+\int_{0}^{t}B(X_{r}^{0,x})dW_{r}~,
\]
for the function $\varphi_{\varepsilon}:H\rightarrow\mathbb{R},$ given by the
formula%
\begin{equation}
\varphi_{\varepsilon}\left(  x\right)  =\varphi(\left(  J_{H}^{\varepsilon
}\right)  ^{1/2}\left(  x\right)  )=\dfrac{1}{2}\left\Vert \left(
J_{H}^{\varepsilon}\right)  ^{1/2}\left(  x\right)  \right\Vert _{H}^{2},
\label{definition of fi}%
\end{equation}
with $\varphi:H\rightarrow\mathbb{R},$ $\varphi\left(  x\right)  =\frac{1}%
{2}\left\Vert x\right\Vert _{H}^{2}.$ In order to make the below estimates
well-defined, we point out that%
\[
B\left(  X\right)  :U\rightarrow H,\quad\left(  B\left(  X\right)  \right)
^{\ast}:H\rightarrow U,\quad\left(  J_{H}^{\varepsilon}\right)  ^{1/2}%
:H\rightarrow H
\]
and $\left(  J_{H}^{\varepsilon}\right)  ^{1/2}\left(  B\left(  X\right)
\right)  :U\rightarrow H$ is a Hilbert-Schmidt operator, this being useful
when we invoke the Burkholder--Davis--Gundy inequality. Indeed, given a basis
$\{g_{k}\}_{k=\overline{1,\infty}}$ in $U$,%
\[%
\begin{array}
[c]{ccl}%
\left\Vert \left(  J_{H}^{\varepsilon}\right)  ^{1/2}\left(  B\left(
X\right)  \right)  \right\Vert _{L_{2}\left(  U;H\right)  }^{2} & = &
%TCIMACRO{\dsum \limits_{k=1}^{\infty}}%
%BeginExpansion
{\displaystyle\sum\limits_{k=1}^{\infty}}
%EndExpansion
\left\Vert \left(  J_{H}^{\varepsilon}\right)  ^{1/2}\left(  B\left(
X\right)  \left(  g_{k}\right)  \right)  \right\Vert _{H}^{2}\medskip\\
& = &
%TCIMACRO{\dsum \limits_{k=1}^{\infty}}%
%BeginExpansion
{\displaystyle\sum\limits_{k=1}^{\infty}}
%EndExpansion
\left\langle \left(  J_{H}^{\varepsilon}\right)  ^{1/2}\left(  B\left(
X\right)  \left(  g_{k}\right)  \right)  ,\left(  J_{H}^{\varepsilon}\right)
^{1/2}\left(  B\left(  X\right)  \left(  g_{k}\right)  \right)  \right\rangle
_{H}\medskip\\
& = &
%TCIMACRO{\dsum \limits_{k=1}^{\infty}}%
%BeginExpansion
{\displaystyle\sum\limits_{k=1}^{\infty}}
%EndExpansion
\left\langle J_{H}^{\varepsilon}(B\left(  X\right)  (g_{k})),B\left(
X\right)  \left(  g_{k}\right)  \right\rangle _{H}\medskip\\
& \leq &
%TCIMACRO{\dsum \limits_{k=1}^{\infty}}%
%BeginExpansion
{\displaystyle\sum\limits_{k=1}^{\infty}}
%EndExpansion
\left\Vert J_{H}^{\varepsilon}(B\left(  X\right)  (g_{k})\right\Vert _{H}%
^{2}\left\Vert B\left(  X\right)  (g_{k})\right\Vert _{H}^{2}.
\end{array}
\]
Since $J_{H}^{\varepsilon}$ is $1$/$\varepsilon-$Lipschitz, we obtain that,
for each fixed $\varepsilon$,%
\[
\left\Vert \left(  J_{H}^{\varepsilon}\right)  ^{1/2}\left(  B\left(
X\right)  \right)  \right\Vert _{L_{2}\left(  U;H\right)  }^{2}\leq\dfrac
{1}{C_{\varepsilon}}%
%TCIMACRO{\dsum \limits_{k=1}^{\infty}}%
%BeginExpansion
{\displaystyle\sum\limits_{k=1}^{\infty}}
%EndExpansion
\left\Vert B\left(  X\right)  (g_{k})\right\Vert _{H}^{2}=\dfrac
{1}{C_{\varepsilon}}\left\Vert B\left(  X\right)  \right\Vert _{L_{2}\left(
U;H\right)  }^{2}\leq C\left(  M_{B},1/\varepsilon\right)
\]

The first and second order differentials are given, respectively, by%
\[
\left\{
\begin{array}
[c]{rcl}%
\left\langle D\varphi_{\varepsilon}\left(  x\right)  ,y\right\rangle _{H} &
= & \left\langle D\varphi\left(  \left(  J_{H}^{\varepsilon}\right)
^{1/2}\left(  x\right)  \right)  ,\left(  J_{H}^{\varepsilon}\right)
^{1/2}\left(  y\right)  \right\rangle _{H}=\left\langle \left(  J_{H}%
^{\varepsilon}\right)  ^{1/2}\left(  x\right)  ,\left(  J_{H}^{\varepsilon
}\right)  ^{1/2}\left(  y\right)  \right\rangle _{H}~,\medskip\\
\left\langle D^{2}\varphi_{\varepsilon}\left(  x\right)  h,k\right\rangle _{H}
& = & D^{2}\varphi\left(  \left(  J_{H}^{\varepsilon}\right)  ^{1/2}\left(
x\right)  \right)  \left\langle \left(  J_{H}^{\varepsilon}\right)
^{1/2}\left(  h\right)  ,\left(  J_{H}^{\varepsilon}\right)  ^{1/2}\left(
k\right)  \right\rangle _{H}\medskip\\
& = & \left\langle \left(  J_{H}^{\varepsilon}\right)  ^{1/2}\left(  h\right)
,\left(  J_{H}^{\varepsilon}\right)  ^{1/2}\left(  k\right)  \right\rangle
_{H}~,
\end{array}
\right.
\]
for every $x,h,k\in H.$ We can observe that%
\[%
\begin{array}
[c]{l}%
\operatorname{Tr}\left[  B^{\ast}\left(  X\right)  D^{2}\varphi_{\varepsilon
}\left(  X\right)  B\left(  X\right)  \right]  \smallskip\\
\quad\quad\quad=%
%TCIMACRO{\dsum \limits_{k=1}^{\infty}}%
%BeginExpansion
{\displaystyle\sum\limits_{k=1}^{\infty}}
%EndExpansion
\left\langle D^{2}\varphi_{\varepsilon}\left(  X\right)  B\left(  X\right)
g_{k},B\left(  X\right)  g_{k}\right\rangle _{H}\smallskip\\
\quad\quad\quad=%
%TCIMACRO{\dsum \limits_{k=1}^{\infty}}%
%BeginExpansion
{\displaystyle\sum\limits_{k=1}^{\infty}}
%EndExpansion
\underbrace{D^{2}\varphi_{\varepsilon}\left(  \left(  J_{H}^{\varepsilon
}\right)  ^{1/2}\left(  X\right)  \right)  }_{=1}\left\langle \left(
J_{H}^{\varepsilon}\right)  ^{1/2}\left(  B\left(  X\right)  \left(
g_{k}\right)  \right)  ,\left(  J_{H}^{\varepsilon}\right)  ^{1/2}\left(
B\left(  X\right)  \left(  g_{k}\right)  \right)  \right\rangle _{H}%
\smallskip\\
\quad\quad\quad=%
%TCIMACRO{\dsum \limits_{k=1}^{\infty}}%
%BeginExpansion
{\displaystyle\sum\limits_{k=1}^{\infty}}
%EndExpansion
\left\Vert \left(  J_{H}^{\varepsilon}\right)  ^{1/2}\left(  B\left(
X\right)  \left(  g_{k}\right)  \right)  \right\Vert _{H}^{2}=\left\Vert
\left(  J_{H}^{\varepsilon}\right)  ^{1/2}\left(  B\left(  X\right)  \right)
\right\Vert _{L_{2}\left(  U;H\right)  }^{2}.
\end{array}
\]
Reminding that $\left(  J_{H}^{\varepsilon}\right)  ^{1/2}$ is self-adjoint,
It\^{o}'s formula leads to%
\begin{equation}%
\begin{array}
[c]{l}%
\dfrac{1}{2}\left\Vert \left(  J_{H}^{\varepsilon}\right)  ^{1/2}(X_{t}%
^{0,x})\right\Vert _{H}^{2}-%
%TCIMACRO{\dint _{0}^{t}}%
%BeginExpansion
{\displaystyle\int_{0}^{t}}
%EndExpansion
\left\langle J_{H}^{\varepsilon}(X_{r}^{0,x}),A(X_{r}^{0,x})\right\rangle
_{H}dr\medskip\\
\quad\quad=\dfrac{1}{2}\left\Vert \left(  J_{H}^{\varepsilon}\right)
^{1/2}\left(  x\right)  \right\Vert _{H}^{2}+%
%TCIMACRO{\dint _{0}^{t}}%
%BeginExpansion
{\displaystyle\int_{0}^{t}}
%EndExpansion
\left\Vert \left(  J_{H}^{\varepsilon}\right)  ^{1/2}(B(X_{r}^{0,x}%
))\right\Vert _{L_{2}\left(  U;H\right)  }^{2}dr\medskip\\
\quad\quad\quad+%
%TCIMACRO{\dint _{0}^{t}}%
%BeginExpansion
{\displaystyle\int_{0}^{t}}
%EndExpansion
\left\langle \left(  J_{H}^{\varepsilon}\right)  ^{1/2}(X_{r}^{0,x}),\left(
J_{H}^{\varepsilon}\right)  ^{1/2}(B(X_{r}^{0,x})dW_{r})\right\rangle _{H},
\end{array}
\label{Ito-Lemma bound in V}%
\end{equation}
We first apply the assumption $(H_{6})$ and the linearity of $B$ in order to
obtain from (\ref{Ito-Lemma bound in V}) the following inequality%
\[%
\begin{array}
[c]{l}%
\dfrac{1}{2}\left\Vert \left(  J_{H}^{\varepsilon}\right)  ^{1/2}(X_{t}%
^{0,x})\right\Vert _{H}^{2}\leq\dfrac{1}{2}\left\Vert \left(  J_{H}%
^{\varepsilon}\right)  ^{1/2}\left(  x\right)  \right\Vert _{H}^{2}+%
%TCIMACRO{\dint _{0}^{t}}%
%BeginExpansion
{\displaystyle\int_{0}^{t}}
%EndExpansion
\left\Vert B(\left(  J_{H}^{\varepsilon}\right)  ^{1/2}(X_{r}^{0,x}%
))\right\Vert _{L_{2}\left(  U;H\right)  }^{2}dr\medskip\\
\quad\quad\quad+%
%TCIMACRO{\dint _{0}^{t}}%
%BeginExpansion
{\displaystyle\int_{0}^{t}}
%EndExpansion
\left\langle \left(  J_{H}^{\varepsilon}\right)  ^{1/2}(X_{r}^{0,x}),\left(
J_{H}^{\varepsilon}\right)  ^{1/2}(B(X_{r}^{0,x})dW_{r})\right\rangle _{H}.
\end{array}
\]
After taking supremum over $t\in\left[  0,s\right]  ,$ with $s>t$ arbitrarily,
fixed, and applying the expectation, it becomes%
\begin{equation}%
\begin{array}
[c]{l}%
\dfrac{1}{2}\mathbb{E}\left(  \sup_{t\in\left[  0,s\right]  }\left\Vert
\left(  J_{H}^{\varepsilon}\right)  ^{1/2}(X_{t}^{0,x})\right\Vert _{H}%
^{2}\right)  \leq\dfrac{1}{2}\mathbb{E}\left\Vert \left(  J_{H}^{\varepsilon
}\right)  ^{1/2}\left(  x\right)  \right\Vert _{H}^{2}+\mathbb{E}%
%TCIMACRO{\dint _{0}^{s}}%
%BeginExpansion
{\displaystyle\int_{0}^{s}}
%EndExpansion
\left\Vert B(\left(  J_{H}^{\varepsilon}\right)  ^{1/2}(X_{r}^{0,x}%
))\right\Vert _{L_{2}\left(  U;H\right)  }^{2}dr\medskip\\
\quad\quad\quad+\mathbb{E}\left(  \sup_{t\in\left[  0,s\right]  }\left\vert
%TCIMACRO{\dint _{0}^{t}}%
%BeginExpansion
{\displaystyle\int_{0}^{t}}
%EndExpansion
\left\langle \left(  J_{H}^{\varepsilon}\right)  ^{1/2}(X_{r}^{0,x}),\left(
J_{H}^{\varepsilon}\right)  ^{1/2}(B(X_{r}^{0,x})dW_{r})\right\rangle
_{H}\right\vert \right)  .
\end{array}
\label{Ito-Lemma bound in V 2}%
\end{equation}
According to the Burkholder--Davis--Gundy inequality (see Pr\'{e}v\^{o}t,
Rockner \cite[Proposition D.0.1]{Prevot/Rockner:07} or, for more refined
constants under various situations, Ren \cite{Ren:08}) and keeping still in
mind the li\-ne\-a\-ri\-ty of $B$, we get, for all $s\in\left[  0,T\right]  $,%
\[%
\begin{array}
[c]{l}%
\mathbb{E}\left(  \sup_{t\in\left[  0,s\right]  }\left\vert
%TCIMACRO{\dint _{0}^{t}}%
%BeginExpansion
{\displaystyle\int_{0}^{t}}
%EndExpansion
\left\langle \left(  J_{H}^{\varepsilon}\right)  ^{1/2}(X_{r}^{0,x}),\left(
J_{H}^{\varepsilon}\right)  ^{1/2}(B(X_{r}^{0,x})dW_{r})\right\rangle
_{H}\right\vert \right)  \medskip\\
\quad\leq3\mathbb{E}\left(
%TCIMACRO{\dint _{0}^{s}}%
%BeginExpansion
{\displaystyle\int_{0}^{s}}
%EndExpansion
\left\Vert (B(X_{r}^{0,x}))^{\ast}\left(  J_{H}^{\varepsilon}\right)
^{1/2}\left(  J_{H}^{\varepsilon}\right)  ^{1/2}(X_{r}^{0,x})\right\Vert
_{U}^{2}dr\right)  ^{1/2}\medskip\\
\quad\leq3\mathbb{E}\left(
%TCIMACRO{\dint _{0}^{s}}%
%BeginExpansion
{\displaystyle\int_{0}^{s}}
%EndExpansion
\left\Vert (\left(  J_{H}^{\varepsilon}\right)  ^{1/2}(B(X_{r}^{0,x})))^{\ast
}\left(  J_{H}^{\varepsilon}\right)  ^{1/2}(X_{r}^{0,x})\right\Vert _{U}%
^{2}dr\right)  ^{1/2}\medskip\\
\quad\leq3\mathbb{E}\left(
%TCIMACRO{\dint _{0}^{s}}%
%BeginExpansion
{\displaystyle\int_{0}^{s}}
%EndExpansion
\left\Vert \left(  J_{H}^{\varepsilon}\right)  ^{1/2}(B(X_{r}^{0,x}%
))\right\Vert _{L_{2}\left(  U;H\right)  }^{2}\left\Vert \left(
J_{H}^{\varepsilon}\right)  ^{1/2}(X_{r}^{0,x})\right\Vert _{H}^{2}dr\right)
^{1/2}\medskip\\
\quad\leq3M_{B}\mathbb{E}\left(
%TCIMACRO{\dint _{0}^{s}}%
%BeginExpansion
{\displaystyle\int_{0}^{s}}
%EndExpansion
\left\Vert \left(  J_{H}^{\varepsilon}\right)  ^{1/2}(X_{r}^{0,x})\right\Vert
_{H}^{2}\left\Vert \left(  J_{H}^{\varepsilon}\right)  ^{1/2}(X_{r}%
^{0,x})\right\Vert _{H}^{2}dr\right)  ^{1/2}\medskip\\
\quad\leq3M_{B}\mathbb{E}\left(  \sup_{t\in\left[  0,s\right]  }\left\Vert
\left(  J_{H}^{\varepsilon}\right)  ^{1/2}(X_{t}^{0,x})\right\Vert _{H}%
^{2}\right)  \left(
%TCIMACRO{\dint _{0}^{s}}%
%BeginExpansion
{\displaystyle\int_{0}^{s}}
%EndExpansion
\left\Vert \left(  J_{H}^{\varepsilon}\right)  ^{1/2}(X_{r}^{0,x})\right\Vert
_{H}^{2}dr\right)  ^{1/2}\medskip\\
\quad\leq\dfrac{1}{4}\mathbb{E}\left(  \sup_{t\in\left[  0,s\right]
}\left\Vert \left(  J_{H}^{\varepsilon}\right)  ^{1/2}(X_{t}^{0,x})\right\Vert
_{H}^{2}\right)  +9M_{B}^{2}\mathbb{E}%
%TCIMACRO{\dint _{0}^{s}}%
%BeginExpansion
{\displaystyle\int_{0}^{s}}
%EndExpansion
\left\Vert \left(  J_{H}^{\varepsilon}\right)  ^{1/2}(X_{r}^{0,x})\right\Vert
_{H}^{2}dr.\medskip
\end{array}
\]
Returning now to (\ref{Ito-Lemma bound in V 2}), we obtain%
\[%
\begin{array}
[c]{l}%
\dfrac{1}{4}\mathbb{E}\left(  \sup_{t\in\left[  0,s\right]  }\left\Vert
\left(  J_{H}^{\varepsilon}\right)  ^{1/2}(X_{t}^{0,x})\right\Vert _{H}%
^{2}\right)  \leq\dfrac{1}{2}\left\Vert \left(  J_{H}^{\varepsilon}\right)
^{1/2}\left(  x\right)  \right\Vert _{H}^{2}\medskip\\
\quad+\mathbb{E}%
%TCIMACRO{\dint _{0}^{s}}%
%BeginExpansion
{\displaystyle\int_{0}^{s}}
%EndExpansion
\left\Vert B(\left(  J_{H}^{\varepsilon}\right)  ^{1/2}(X_{r}^{0,x}%
))\right\Vert _{L_{2}\left(  U;H\right)  }^{2}dr+9\mathbb{E}%
%TCIMACRO{\dint _{0}^{s}}%
%BeginExpansion
{\displaystyle\int_{0}^{s}}
%EndExpansion
\left\Vert \left(  J_{H}^{\varepsilon}\right)  ^{1/2}(X_{r}^{0,x})\right\Vert
_{H}^{2}dr\medskip\\
\leq\dfrac{1}{2}\left\Vert \left(  J_{H}^{\varepsilon}\right)  ^{1/2}\left(
x\right)  \right\Vert _{H}^{2}+10M_{B}^{2}\mathbb{E}%
%TCIMACRO{\dint _{0}^{s}}%
%BeginExpansion
{\displaystyle\int_{0}^{s}}
%EndExpansion
\left\Vert \left(  J_{H}^{\varepsilon}\right)  ^{1/2}(X_{r}^{0,x})\right\Vert
_{H}^{2}dr.
\end{array}
\]
Classical use of Gronwall's inequality provides now that%
\begin{equation}%
\begin{array}
[c]{ccc}%
\mathbb{E}\sup_{t\in\left[  0,s\right]  }\left\Vert \left(  J_{H}%
^{\varepsilon}\right)  ^{1/2}(X_{t}^{0,x})\right\Vert _{H}^{2} & \leq & C
\end{array}
\label{boundedness of supremum}%
\end{equation}
and, passing to the limit, as $\varepsilon\rightarrow0$, we obtain the
inequality (\ref{marg-sol}), due to relation (\ref{norm equality}%
).$\smallskip$

We can now advance to the proof of the inequality (\ref{marg-dif}). Starting
from%
\[
X_{r}^{0,x}-X_{t}^{0,x}=%
%TCIMACRO{\dint _{t}^{r}}%
%BeginExpansion
{\displaystyle\int_{t}^{r}}
%EndExpansion
A(X_{u}^{0,x})du+\int_{t}^{r}B(X_{u}^{0,x})dW_{u}~,
\]
apply It\^{o}'s formula for the function $\varphi_{\varepsilon}$ given by
(\ref{definition of fi}). We have%
\[%
\begin{array}
[c]{l}%
\dfrac{1}{2}\left\Vert \left(  J_{H}^{\varepsilon}\right)  ^{1/2}(X_{r}%
^{0,x}-X_{t}^{0,x})\right\Vert _{H}^{2}-%
%TCIMACRO{\dint _{t}^{r}}%
%BeginExpansion
{\displaystyle\int_{t}^{r}}
%EndExpansion
\left\langle J_{H}^{\varepsilon}(X_{u}^{0,x}),A(X_{u}^{0,x})\right\rangle
_{H}du\medskip\\
\quad\quad=%
%TCIMACRO{\dint _{t}^{r}}%
%BeginExpansion
{\displaystyle\int_{t}^{r}}
%EndExpansion
\left\Vert \left(  J_{H}^{\varepsilon}\right)  ^{1/2}(B(X_{u}^{0,x}%
))\right\Vert _{L_{2}\left(  U;H\right)  }^{2}du+%
%TCIMACRO{\dint _{t}^{r}}%
%BeginExpansion
{\displaystyle\int_{t}^{r}}
%EndExpansion
\left\langle \left(  J_{H}^{\varepsilon}\right)  ^{1/2}(X_{u}^{0,x}),\left(
J_{H}^{\varepsilon}\right)  ^{1/2}(B(X_{u}^{0,x})dW_{u})\right\rangle _{H}%
\end{array}
\]
Once again, we use the hypothesis $(H_{6}),$ take $h>0,$ consider the supremum
over $r\in\left[  t,t+h\right]  $ and apply the expectation, in order to
obtain%
\begin{equation}%
\begin{array}
[c]{l}%
\dfrac{1}{2}\mathbb{E}\left(  \sup_{r\in\left[  t,t+h\right]  }\left\Vert
\left(  J_{H}^{\varepsilon}\right)  ^{1/2}(X_{r}^{0,x}-X_{t}^{0,x})\right\Vert
_{H}^{2}\right)  \leq\mathbb{E}%
%TCIMACRO{\dint _{t}^{t+h}}%
%BeginExpansion
{\displaystyle\int_{t}^{t+h}}
%EndExpansion
\left\Vert B(\left(  J_{H}^{\varepsilon}\right)  ^{1/2}(X_{u}^{0,x}%
))\right\Vert _{L_{2}\left(  U;H\right)  }^{2}du\medskip\\
\quad\quad\quad+\mathbb{E}\left(  \sup_{r\in\left[  t,t+h\right]  }\left\vert
%TCIMACRO{\dint _{t}^{r}}%
%BeginExpansion
{\displaystyle\int_{t}^{r}}
%EndExpansion
\left\langle \left(  J_{H}^{\varepsilon}\right)  ^{1/2}(X_{u}^{0,x}),\left(
J_{H}^{\varepsilon}\right)  ^{1/2}(B(X_{u}^{0,x})dW_{u})\right\rangle
_{H}\right\vert \right)  .
\end{array}
\label{Ito for the increments}%
\end{equation}
We have, according to assumption $\left(  H_{7}\right)  $ and to the
boundedness (\ref{boundedness of supremum}),%
\begin{align*}
\mathbb{E}%
%TCIMACRO{\dint _{t}^{t+h}}%
%BeginExpansion
{\displaystyle\int_{t}^{t+h}}
%EndExpansion
\left\Vert B(\left(  J_{H}^{\varepsilon}\right)  ^{1/2}(X_{u}^{0,x}%
))\right\Vert _{L_{2}\left(  U;H\right)  }^{2}du  &  \leq M_{B}\mathbb{E}%
%TCIMACRO{\dint _{t}^{t+h}}%
%BeginExpansion
{\displaystyle\int_{t}^{t+h}}
%EndExpansion
\left\Vert \left(  J_{H}^{\varepsilon}\right)  ^{1/2}(X_{u}^{0,x})\right\Vert
_{H}^{2}du\\
&  \leq M_{B}h\mathbb{E}\left(  \sup\nolimits_{r\in\left[  t,t+h\right]
}\left\Vert \left(  J_{H}^{\varepsilon}\right)  ^{1/2}(X_{r}^{0,x})\right\Vert
_{H}^{2}\right)  \leq Ch
\end{align*}
By the Burkholder--Davis--Gundy inequality it follows once again, applying
(\ref{boundedness of supremum}),%
\[%
\begin{array}
[c]{l}%
\mathbb{E}\left(  \sup_{r\in\left[  t,t+h\right]  }\left\vert
%TCIMACRO{\dint _{t}^{r}}%
%BeginExpansion
{\displaystyle\int_{t}^{r}}
%EndExpansion
\left\langle \left(  J_{H}^{\varepsilon}\right)  ^{1/2}(X_{u}^{0,x}),\left(
J_{H}^{\varepsilon}\right)  ^{1/2}(B(X_{u}^{0,x})dW_{u})\right\rangle
_{H}\right\vert \right)  \medskip\\
\quad\leq3M_{B}\mathbb{E}\left(  \sup_{r\in\left[  t,t+h\right]  }\left\Vert
\left(  J_{H}^{\varepsilon}\right)  ^{1/2}(X_{t}^{0,x})\right\Vert _{H}%
^{2}\right)  \left(
%TCIMACRO{\dint _{t}^{t+h}}%
%BeginExpansion
{\displaystyle\int_{t}^{t+h}}
%EndExpansion
\left\Vert \left(  J_{H}^{\varepsilon}\right)  ^{1/2}(X_{u}^{0,x})\right\Vert
_{H}^{2}du\right)  ^{1/2}\leq3M_{B}C^{2}h
\end{array}
\]
Insert the last to inequalities into (\ref{Ito for the increments}) and the
conclusion follows.\hfill\medskip
\end{proof}

\begin{remark}
We highlight that, the a priori estimate (\ref{marg-sol}) obtained in Lemma
\ref{estim1} ensures that, for each trajectory, there exists a finite positive
constant radius $R$ such that the solution remains in the ball $B_{R}\left(
0\right)  \subset V.$ For the purpose of our research, we limit our interest
only on the locally Lipschitz constant $L_{A,R},$ provided by the hypotheses
$\left(  H_{5}\right)  $. However, in order to keep the notation light, we
write $L_{A}$ throughout, but it should be understood pathwise as the local
Lipschitz constant on the visited bounded set.
\end{remark}

We advance now to the proof of our main viability result, Theorem
\ref{main result}.\medskip

\begin{proof}
[Proof of Theorem \ref{main result}]$\smallskip$

\noindent\fbox{\textbf{Step I }$(i)\Rightarrow(ii):$}\textbf{ }We will start
by proving that if $X$ is a viable stochastic process in $K$ then, for every
$t\in\left[  0,T\right]  $ and for every $\mathcal{F}_{t}$ - random variable
$x\in L^{2}\left(  \Omega;V\right)  $ from $K$ we have that%
\[
\left(  A\left(  x\right)  ,B\left(  x\right)  \right)  \in\mathcal{S}%
_{K}\left(  t,x\right)  .
\]

Let $t\in\left[  0,T\right]  $ and $x\in L^{2}\left(  \Omega;V\right)  $ an
$\mathcal{F}_{t}$ - random variable from $K.$

Under the assumptions above, we first apply Theorem 1.1 at page 83 from
\cite{Pardoux:75} and we get that equation (\ref{GeneralFormEq.}) has a unique
strong solution $X_{s}^{t,x}$,\ which belongs to $K$ by the viability property.

We take the solution%
\[
X_{s}^{t,x}=x+\int_{t}^{s}A\left(  X_{r}^{t,x}\right)  dr+\int_{t}^{s}B\left(
X_{r}^{t,x}\right)  dW_{r}\in K
\]
and we rewrite as%
\[
X_{t+h}^{t,x}=x+\int_{t}^{t+h}\left(  A\left(  X_{r}^{t,x}\right)  -A\left(
x\right)  +A\left(  x\right)  \right)  dr+\int_{t}^{t+h}\left(  B\left(
X_{r}^{t,x}\right)  -B\left(  x\right)  +B\left(  x\right)  \right)  dW_{r}\in
K.
\]
In order to get that
\[
\left(  A\left(  x\right)  ,B\left(  x\right)  \right)  \in\mathcal{S}%
_{K}\left(  t,x\right)  ,
\]
it is sufficient to show that
\[
\alpha\left(  r\right)  =A\left(  X_{r}^{t,x}\right)  -A\left(  x\right)
\text{ converges to }0\text{ in }V^{\ast}\text{ for }r\rightarrow t
\]
and%
\[
\beta\left(  r\right)  =B\left(  X_{r}^{t,x}\right)  -B\left(  x\right)
\text{ converges to }0\text{ in }L_{2}\left(  U;H\right)  \text{ for
}r\rightarrow t.
\]

This follows directly, by using the Lipschitz assumptions on $A$ and $B$
combined with the Lemma \ref{estim1}. Indeed,%
\[
\left\{
\begin{array}
[c]{l}%
\left\Vert A(X_{r}^{t,x})-A\left(  x\right)  \right\Vert _{V^{\ast}}\leq
L_{A}\left\Vert X_{r}^{t,x}-x\right\Vert _{V}\smallskip\\
\left\Vert B(X_{r}^{t,x})-B\left(  x\right)  \right\Vert _{L_{2}\left(
U;H\right)  }\leq\mathcal{C}_{V,H}M_{B}\left\Vert X_{r}^{t,x}-x\right\Vert
_{V}.
\end{array}
\right.
\]
By taking the supremum over $r\in\left[  t,t+h\right]  $ and the expectation,
we obtain%
\begin{align*}
\mathbb{E}\left(  \sup\limits_{r\in\left[  t,t+h\right]  }\left\Vert
\alpha(r)\right\Vert _{V^{\ast}}^{2}\right)   &  =\mathbb{E}\left(
\sup\limits_{r\in\left[  t,t+h\right]  }\left\Vert A\left(  X_{r}%
^{t,x}\right)  -A\left(  x\right)  \right\Vert _{V^{\ast}}^{2}\right)
\medskip\\
&  \leq L_{A}^{2}\mathbb{E}\left(  \sup\limits_{r\in\left[  t,t+h\right]
}\left\Vert X_{r}^{t,x}-x\right\Vert _{V}^{2}\right)  \leq L_{A}^{2}Ch
\end{align*}
and%
\begin{align*}
\mathbb{E}\left(  \sup\limits_{r\in\left[  t,t+h\right]  }\left\Vert \frac
{1}{h^{\delta}}\beta\left(  r\right)  \right\Vert _{L_{2}\left(  U;H\right)
}^{2}\right)   &  =\frac{1}{h^{2\delta}}\mathbb{E}\left(  \sup\limits_{r\in
\left[  t,t+h\right]  }\left\Vert B\left(  X_{r}^{t,x}\right)  -B\left(
x\right)  \right\Vert _{L_{2}\left(  U;H\right)  }^{2}\right)  \medskip\\
&  \leq\frac{\mathcal{C}_{V,H}^{2}}{h^{2\delta}}M_{B}^{2}\mathbb{E}\left(
\sup\limits_{r\in\left[  t,t+h\right]  }\left\Vert X_{r}^{t,x}-x\right\Vert
_{V}^{2}\right)  \leq M_{B}^{2}Ch^{1-2\delta}.
\end{align*}
The first assertion is now complete.$\smallskip$

\noindent\fbox{\textbf{Step II }$(ii)\Rightarrow(iii):$\textbf{ }} It will be
shown that%
\[
\mathcal{S}_{K}\left(  t,x\right)  \subset\mathcal{T}_{K}\left(  t,x\right)
.
\]

We set%
\[
\alpha_{h}:=\frac{1}{h}\int_{t}^{t+h}\alpha(s)ds\quad\quad\text{and}\quad
\quad\beta_{h}=\frac{1}{h^{\frac{1}{2}+\delta}}\int_{t}^{t+h}\beta(s)dW_{s}.
\]
We deduce from the Cauchy-Schwarz inequality that
\begin{align*}
\mathbb{E}\left(  \left\Vert \alpha_{h}\right\Vert _{V^{\ast}}^{2}\right)   &
=\mathbb{E}\left(  \left\Vert \frac{1}{h}\int_{t}^{t+h}\alpha(s)ds\right\Vert
_{V^{\ast}}^{2}\right)  =\frac{1}{h^{2}}\mathbb{E}\left(  \left\Vert \int
_{t}^{t+h}1\cdot\alpha(s)ds\right\Vert _{V^{\ast}}^{2}\right) \\
&  \leq\frac{1}{h^{2}}\mathbb{E}\left(  h\int_{t}^{t+h}\left\Vert
\alpha(s)\right\Vert _{V^{\ast}}^{2}ds\right)  \leq\frac{1}{h}\mathbb{E}%
%TCIMACRO{\dint _{t}^{t+h}}%
%BeginExpansion
{\displaystyle\int_{t}^{t+h}}
%EndExpansion
\sup\limits_{r\in\left[  t,t+h\right]  }\left(  \left\Vert \alpha
(r)\right\Vert _{V\ast}^{2}\right)  ds\\
&  \leq\mathbb{E}\left(  \sup\limits_{r\in\left[  t,t+h\right]  }\left(
\left\Vert \alpha(r)\right\Vert _{V\ast}^{2}\right)  \right)
\end{align*}
which converges to 0, for $h\rightarrow0$, due to the choice of $\alpha$ from
Definition \ref{defS}.

Similarly, using the Burkholder-Davis-Gundy inequality,
\begin{align*}
\mathbb{E}\left(  \left\Vert \beta_{h}\right\Vert _{H}^{2}\right)   &
=\mathbb{E}\left(  \left\Vert \frac{1}{h^{\frac{1}{2}+\delta}}\int_{t}%
^{t+h}\beta(s)dW_{s}\right\Vert _{H}^{2}\right)  \leq\frac{1}{h^{1+2\delta}%
}\mathbb{E}\left(  \sup_{s\in\left[  t,t+h\right]  }\left\Vert \int_{t}%
^{s}\beta(s)dW_{s}\right\Vert _{H}^{2}\right) \\
&  \leq\frac{C}{h^{1+2\delta}}\mathbb{E}\left(  \int_{t}^{t+h}\left\Vert
\beta\left(  s\right)  \right\Vert _{L_{2}\left(  U;H\right)  }^{2}ds\right)
\leq\frac{C}{h^{1+2\delta}}\mathbb{E}\left(  \int_{t}^{t+h}\sup_{r\in\left[
t,t+h\right]  }\left\Vert \beta\left(  r\right)  \right\Vert _{L_{2}\left(
U;H\right)  }^{2}ds\right) \\
&  \leq C\mathbb{E}\left(  \sup_{r\in\left[  t,t+h\right]  }\left\Vert
\frac{1}{h^{\delta}}\beta\left(  r\right)  \right\Vert _{L_{2}\left(
U;H\right)  }^{2}\right)  ,\quad\text{for some }C>0,
\end{align*}
which converges to 0, for $h\rightarrow0$. The expectation$\ $of \ $\beta_{h}$
is obviously equal to $0$, for every $h,$ and $\beta_{h}$ is independent of
$\ \mathcal{F}_{t}.$

\fbox{\textbf{Step III }$(iii)\Rightarrow(i):$ } We extend $A$ and $B$ defined
on $K$ by the maps $\tilde{A}$ and $\tilde{B}$ defined on the whole space by%
\[
\tilde{A}\left(  x\right)  :=A\left(  \pi_{K}\left(  x\right)  \right)
\quad\text{and}\quad\tilde{B}:=B\left(  \pi_{K}\left(  x\right)  \right)  .
\]
The pair $(\tilde{A},\tilde{B})$ satisfies, obviously, the condition%
\[
(\tilde{A}\left(  x\right)  ,\tilde{B}\left(  x\right)  )\in\mathcal{T}%
_{K}\left(  t,\pi_{K}\left(  x\right)  \right)  .
\]
Starting from this point, we use some specific technical instruments, in order
to be able to control, on the right hand side, a scaled variation of the
expected value applied to the square distance of the state to the viability
set. This implies the viability of $K$ with respect to $(\tilde{A},\tilde
{B}).$ Since these maps coincide with $\left(  A,B\right)  $ on $K$, we deduce
that $K$ is a viability domain of $\left(  A,B\right)  .$

Let consider now as the initial datum an $\mathcal{F}_{t}-$measurable random
variable $x\in V$ and select an $\mathcal{F}_{t}$-measurable projection
$y\in\Pi_{K}\left(  x\right)  \subset V,$ $\mathbb{P-}a.s.$ By omitting the
tilde sign from the pair $(\tilde{A},\tilde{B})$ - for the simplicity of the
presentation - one can write, due to the nature of the above extension,%
\[
\left(  A\left(  x\right)  ,B\left(  x\right)  \right)  \in\mathcal{T}%
_{K}\left(  t,y\right)  .
\]
We write the solution of Eq.(\ref{GeneralFormEq.}) under integral form, on
$\left[  t,t+h\right]  ,$ for $h\geq0,$ small enough:%
\[
\left\{
\begin{array}
[c]{l}%
X_{t+h}=X_{t}+%
%TCIMACRO{\dint _{t}^{t+h}}%
%BeginExpansion
{\displaystyle\int_{t}^{t+h}}
%EndExpansion
A\left(  X_{s}\right)  ds+%
%TCIMACRO{\dint _{t}^{t+h}}%
%BeginExpansion
{\displaystyle\int_{t}^{t+h}}
%EndExpansion
B\left(  X_{s}\right)  dW_{s}~,\medskip\\
X_{t}=x.
\end{array}
\right.
\]
Chose, in an arbitrarily manner, any pair $\left(  \gamma,\nu\right)  $ of
$\mathcal{F}_{t}-$measurable random variables, belonging to the variational
contingent set $\mathcal{T}_{K}\left(  t,y\right)  $. From Definition
\ref{defT} it follows the existence of a positive sequence $\left(
h_{n}\right)  _{n}$, converging to $0$, and of some $\mathcal{F}_{t+h_{n}}%
-$measurable sequences $\left(  \alpha_{n}\right)  _{n}$ and $\left(
\beta_{n}\right)  _{n}$ satisfying (\ref{ConditiiAlfaBeta}). Moreover, for
every $n\geq0,$%
\[
y+\nu\left(  W_{t+h_{n}}-W_{t}\right)  +\gamma h_{n}+h_{n}\alpha_{n}%
+h_{n}^{\frac{1}{2}+\delta}\beta_{n}\in K,\quad\mathbb{P-}a.s.
\]
Therefore%
\begin{equation}
\left\{
\begin{array}
[c]{l}%
d_{K}^{2}\left(  X_{t+h_{n}}\right)  -d_{K}^{2}\left(  X_{t}\right)
\medskip\\
\leq\left\Vert x+%
%TCIMACRO{\dint _{t}^{t+h_{n}}}%
%BeginExpansion
{\displaystyle\int_{t}^{t+h_{n}}}
%EndExpansion
A\left(  X_{s}\right)  ds+%
%TCIMACRO{\dint _{t}^{t+h_{n}}}%
%BeginExpansion
{\displaystyle\int_{t}^{t+h_{n}}}
%EndExpansion
B\left(  X_{s}\right)  dW_{s}-y-\nu\left(  W_{t+h_{n}}-W_{t}\right)  -\gamma
h_{n}-h_{n}\alpha_{n}-h_{n}^{\frac{1}{2}+\delta}\beta_{n}\right\Vert
_{V^{\ast}}^{2}\medskip\\
\quad-\left\Vert x-y\right\Vert _{V^{\ast}}^{2}\medskip\\
\leq\left\Vert \left(  x-y\right)  +%
%TCIMACRO{\dint _{t}^{t+h_{n}}}%
%BeginExpansion
{\displaystyle\int_{t}^{t+h_{n}}}
%EndExpansion
\left(  A\left(  X_{s}\right)  -\gamma\right)  ds+%
%TCIMACRO{\dint _{t}^{t+h_{n}}}%
%BeginExpansion
{\displaystyle\int_{t}^{t+h_{n}}}
%EndExpansion
\left(  B\left(  X_{s}\right)  -\nu\right)  dW_{s}-h_{n}\alpha_{n}%
-h_{n}^{\frac{1}{2}+\delta}\beta_{n}\right\Vert _{V^{\ast}}^{2}-\left\Vert
x-y\right\Vert _{V^{\ast}}^{2}\medskip\\
=:I
\end{array}
\right.  \label{delta_distance}%
\end{equation}
The last term can be split in the following way:%
\[
\left\{
\begin{array}
[c]{lll}%
I= & 2\left\langle x-y,%
%TCIMACRO{\dint _{t}^{t+h_{n}}}%
%BeginExpansion
{\displaystyle\int_{t}^{t+h_{n}}}
%EndExpansion
\left(  B\left(  X_{s}\right)  -\nu\right)  dW_{s}\right\rangle _{V^{\ast}%
}\medskip & \qquad I_{1}\\
& +2\left\langle x-y,%
%TCIMACRO{\dint _{t}^{t+h_{n}}}%
%BeginExpansion
{\displaystyle\int_{t}^{t+h_{n}}}
%EndExpansion
\left(  A\left(  X_{s}\right)  -\gamma\right)  ds\right\rangle _{V^{\ast}%
}\medskip & \qquad I_{2}\\
& +\left\Vert
%TCIMACRO{\dint _{t}^{t+h_{n}}}%
%BeginExpansion
{\displaystyle\int_{t}^{t+h_{n}}}
%EndExpansion
\left(  B\left(  X_{s}\right)  -\nu\right)  dW_{s}\right\Vert _{V^{\ast}}%
^{2}\medskip & \qquad I_{3}\\
& +\left\Vert
%TCIMACRO{\dint _{t}^{t+h_{n}}}%
%BeginExpansion
{\displaystyle\int_{t}^{t+h_{n}}}
%EndExpansion
\left(  A\left(  X_{s}\right)  -\gamma\right)  ds\right\Vert _{V^{\ast}}%
^{2}\medskip & \qquad I_{4}\\
& +2\left\langle
%TCIMACRO{\dint _{t}^{t+h_{n}}}%
%BeginExpansion
{\displaystyle\int_{t}^{t+h_{n}}}
%EndExpansion
\left(  B\left(  X_{s}\right)  -\nu\right)  dW_{s},%
%TCIMACRO{\dint _{t}^{t+h_{n}}}%
%BeginExpansion
{\displaystyle\int_{t}^{t+h_{n}}}
%EndExpansion
\left(  A\left(  X_{s}\right)  -\gamma\right)  ds\right\rangle _{V^{\ast}%
}\medskip & \qquad I_{5}\\
& -2\left\langle x-y+%
%TCIMACRO{\dint _{t}^{t+h_{n}}}%
%BeginExpansion
{\displaystyle\int_{t}^{t+h_{n}}}
%EndExpansion
\left(  A\left(  X_{s}\right)  -\gamma\right)  ds+%
%TCIMACRO{\dint _{t}^{t+h_{n}}}%
%BeginExpansion
{\displaystyle\int_{t}^{t+h_{n}}}
%EndExpansion
\left(  B\left(  X_{s}\right)  -\nu\right)  dW_{s},h_{n}\alpha_{n}%
\right\rangle _{V^{\ast}}\medskip & \qquad I_{6}\\
& -2\left\langle x-y+%
%TCIMACRO{\dint _{t}^{t+h_{n}}}%
%BeginExpansion
{\displaystyle\int_{t}^{t+h_{n}}}
%EndExpansion
\left(  A\left(  X_{s}\right)  -\gamma\right)  ds+%
%TCIMACRO{\dint _{t}^{t+h_{n}}}%
%BeginExpansion
{\displaystyle\int_{t}^{t+h_{n}}}
%EndExpansion
\left(  B\left(  X_{s}\right)  -\nu\right)  dW_{s},h_{n}^{\frac{1}{2}+\delta
}\beta_{n}\right\rangle _{V^{\ast}}\medskip & \qquad I_{7}\\
& +\left\Vert h_{n}\alpha_{n}+h_{n}^{\frac{1}{2}+\delta}\beta_{n}\right\Vert
_{V^{\ast}}^{2} & \qquad I_{8}%
\end{array}
\right.
\]

We take the expectation in the inequality (\ref{delta_distance}) and estimate
each term of the right-hand side. First, we point out that%
\[
\mathbb{E}\left(  \left\langle x-y,%
%TCIMACRO{\dint _{t}^{t+h_{n}}}%
%BeginExpansion
{\displaystyle\int_{t}^{t+h_{n}}}
%EndExpansion
\left(  B\left(  X_{s}\right)  -\nu\right)  dW_{s}\right\rangle _{V^{\ast}%
}\right)  =\mathbb{E}\left(
%TCIMACRO{\dint _{t}^{t+h_{n}}}%
%BeginExpansion
{\displaystyle\int_{t}^{t+h_{n}}}
%EndExpansion
\left\langle x-y,\left(  B\left(  X_{s}\right)  -\nu\right)  dW_{s}%
\right\rangle _{V^{\ast}}\right)  =0.
\]
Indeed, let us consider the process $Y$ given by%
\[
dY_{t}=\left(  B\left(  X_{t}\right)  -\nu\right)  dW_{t}%
\]
and the the mapping $Z:V^{\ast}\rightarrow\mathbb{R},$%
\[
Z\left(  X\right)  =\langle x-y,X\rangle_{V^{\ast}}~.
\]
By It\^{o}'s formula, it follows that%
\[%
\begin{array}
[c]{rcl}%
dZ\left(  Y_{t}\right)  & = & Z^{\prime}\left(  Y_{t}\right)  dY\left(
t\right)  +\dfrac{1}{2}Z^{\prime\prime}\left(  Y_{t}\right)  \left(
dY_{t}\right)  ^{2}\smallskip\\
dZ\left(  Y_{t}\right)  & = & \langle x-y,\left(  B\left(  X_{t}\right)
-\nu\right)  dW_{t}\rangle_{V^{\ast}}~,
\end{array}
\]
that gives%
\[
\left\langle x-y,%
%TCIMACRO{\dint _{t}^{t+h_{n}}}%
%BeginExpansion
{\displaystyle\int_{t}^{t+h_{n}}}
%EndExpansion
\left(  B\left(  X_{s}\right)  -\nu\right)  dW_{s}\right\rangle _{V^{\ast}}=%
%TCIMACRO{\dint _{t}^{t+h_{n}}}%
%BeginExpansion
{\displaystyle\int_{t}^{t+h_{n}}}
%EndExpansion
\left\langle x-y,\left(  B\left(  X_{s}\right)  -\nu\right)  dW_{s}%
\right\rangle _{V^{\ast}}%
\]

For $I_{2}$ we have%
\[
\mathbb{E}\left(  I_{2}\right)  =\mathbb{E}\left(  2\left\langle x-y,%
%TCIMACRO{\dint _{t}^{t+h_{n}}}%
%BeginExpansion
{\displaystyle\int_{t}^{t+h_{n}}}
%EndExpansion
\left(  A\left(  X_{s}\right)  -\gamma\right)  ds\right\rangle _{V^{\ast}%
}\right)  =2h_{n}\mathbb{E}\left(  \left\langle x-y,\frac{1}{h_{n}}%
%TCIMACRO{\dint _{t}^{t+h_{n}}}%
%BeginExpansion
{\displaystyle\int_{t}^{t+h_{n}}}
%EndExpansion
\left(  A\left(  X_{s}\right)  -\gamma\right)  ds\right\rangle _{V^{\ast}%
}\right)
\]
and by Evans \cite[Appendix A, Theorem 8]{Evans:10}, it follows that%
\[
\mathbb{E}\left(  \left\langle x-y,\dfrac{1}{h_{n}}%
%TCIMACRO{\dint _{t}^{t+h_{n}}}%
%BeginExpansion
{\displaystyle\int_{t}^{t+h_{n}}}
%EndExpansion
\left(  A\left(  X_{s}\right)  -\gamma\right)  ds\right\rangle _{V^{\ast}%
}\right)  =\mathbb{E}\left(  \dfrac{1}{h_{n}}%
%TCIMACRO{\dint _{t}^{t+h_{n}}}%
%BeginExpansion
{\displaystyle\int_{t}^{t+h_{n}}}
%EndExpansion
\left\langle x-y,A\left(  X_{s}\right)  -\gamma\right\rangle _{V^{\ast}%
}ds\right)  .
\]
For $h_{n}\rightarrow0$ , the expectation converges to $\mathbb{E}\left(
\langle x-y,\left(  A\left(  X_{t}\right)  -\gamma\right)  \rangle_{V^{\ast}%
}\right)  $.

Consider now the third term%
\[
\mathbb{E}\left(  \left\Vert
%TCIMACRO{\dint _{t}^{t+h_{n}}}%
%BeginExpansion
{\displaystyle\int_{t}^{t+h_{n}}}
%EndExpansion
\left(  B\left(  X_{s}\right)  -\nu\right)  dW_{s}\right\Vert _{V^{\ast}}%
^{2}\right)  =h_{n}\left(  \dfrac{1}{h_{n}}\mathbb{E}\left(  \left\Vert
%TCIMACRO{\dint _{t}^{t+h_{n}}}%
%BeginExpansion
{\displaystyle\int_{t}^{t+h_{n}}}
%EndExpansion
\left(  B\left(  X_{s}\right)  -\nu\right)  dW_{s}\right\Vert _{V^{\ast}}%
^{2}\right)  \right)
\]
and, using the Burkholder-Davis-Gundy inequality, we obtain%
\[%
\begin{array}
[c]{ccl}%
\dfrac{1}{h_{n}}\mathbb{E}\left(  \left\Vert
%TCIMACRO{\dint _{t}^{t+h_{n}}}%
%BeginExpansion
{\displaystyle\int_{t}^{t+h_{n}}}
%EndExpansion
\left(  B\left(  X_{s}\right)  -\nu\right)  dW_{s}\right\Vert _{V^{\ast}}%
^{2}\right)  & \leq & \dfrac{1}{h_{n}}\mathbb{E}\left(  \sup\limits_{s\in
\left[  t,t+h_{n}\right]  }\left\Vert
%TCIMACRO{\dint _{t}^{s}}%
%BeginExpansion
{\displaystyle\int_{t}^{s}}
%EndExpansion
\left(  B\left(  X_{r}\right)  -\nu\right)  dW_{r}\right\Vert _{V^{\ast}}%
^{2}\right)  \medskip\\
& \leq & \dfrac{3\mathcal{C}_{H,V^{\ast}}^{2}}{h_{n}}\mathbb{E}\left(
%TCIMACRO{\dint _{t}^{t+h_{n}}}%
%BeginExpansion
{\displaystyle\int_{t}^{t+h_{n}}}
%EndExpansion
\left\Vert B\left(  X_{s}\right)  -\nu\right\Vert _{L^{2}\left(  U;H\right)
}^{2}ds\right)  ,
\end{array}
\]
which converges to $3\mathbb{E}(\left\Vert B\left(  X_{t}\right)
-\nu\right\Vert _{L^{2}\left(  U;H\right)  }^{2}),$ as $h_{n}\rightarrow0,$ by
the Lebesgue's dominated convergence theorem.

Due to the Cauchy-Schwarz inequality, the term $I_{4}$ can be estimated as%
\[%
\begin{array}
[c]{l}%
\mathbb{E}\left(  \left\Vert
%TCIMACRO{\dint _{t}^{t+h_{n}}}%
%BeginExpansion
{\displaystyle\int_{t}^{t+h_{n}}}
%EndExpansion
\left(  A\left(  X_{s}\right)  -\gamma\right)  ds\right\Vert _{V^{\ast}}%
^{2}\right)  =h_{n}\left(  \dfrac{1}{h_{n}}\mathbb{E}\left(  \left\Vert
%TCIMACRO{\dint _{t}^{t+h_{n}}}%
%BeginExpansion
{\displaystyle\int_{t}^{t+h_{n}}}
%EndExpansion
1\cdot\left(  A\left(  X_{s}\right)  -\gamma\right)  ds\right\Vert _{V^{\ast}%
}^{2}\right)  \right)  \medskip\\
\quad\quad\leq h_{n}\mathbb{E}\left(
%TCIMACRO{\dint _{t}^{t+h_{n}}}%
%BeginExpansion
{\displaystyle\int_{t}^{t+h_{n}}}
%EndExpansion
\left\Vert A\left(  X_{s}\right)  -\gamma\right\Vert _{V^{\ast}}^{2}ds\right)
\medskip\\
\quad\quad\leq h_{n}\mathbb{E}\left(
%TCIMACRO{\dint _{t}^{t+h_{n}}}%
%BeginExpansion
{\displaystyle\int_{t}^{t+h_{n}}}
%EndExpansion
\sup\limits_{r\in\left[  t,t+h_{n}\right]  }\left\Vert A\left(  X_{r}\right)
-\gamma\right\Vert _{V^{\ast}}^{2}ds\right)  \leq h_{n}^{2}\mathbb{E}\left(
\sup\limits_{r\in\left[  t,t+h_{n}\right]  }\left\Vert A\left(  X_{r}\right)
-\gamma\right\Vert _{V^{\ast}}^{2}\right)  \leq h_{n}^{2}K,
\end{array}
\]
since $A\left(  X_{s}\right)  $ and $\gamma$ are square-integrable.$\smallskip
$

Similarly, for $I_{5}$, we obtain%
\[%
\begin{array}
[c]{l}%
2\mathbb{E}\left(  \left\langle
%TCIMACRO{\dint _{t}^{t+h_{n}}}%
%BeginExpansion
{\displaystyle\int_{t}^{t+h_{n}}}
%EndExpansion
\left(  B\left(  X_{s}\right)  -\nu\right)  dW_{s},%
%TCIMACRO{\dint _{t}^{t+h_{n}}}%
%BeginExpansion
{\displaystyle\int_{t}^{t+h_{n}}}
%EndExpansion
\left(  A\left(  X_{s}\right)  -\gamma\right)  ds\right\rangle _{V^{\ast}%
}\right)  \medskip\\
\quad\leq2\left(  \mathbb{E}\left(  \left\Vert
%TCIMACRO{\dint _{t}^{t+h_{n}}}%
%BeginExpansion
{\displaystyle\int_{t}^{t+h_{n}}}
%EndExpansion
\left(  B\left(  X_{s}\right)  -\nu\right)  dW_{s}\right\Vert _{V^{\ast}}%
^{2}\right)  \right)  ^{1/2}\left(  \mathbb{E}\left(  \left\Vert
%TCIMACRO{\dint _{t}^{t+h_{n}}}%
%BeginExpansion
{\displaystyle\int_{t}^{t+h_{n}}}
%EndExpansion
\left(  A\left(  X_{s}\right)  -\gamma\right)  ds\right\Vert _{V^{\ast}}%
^{2}\right)  \right)  ^{1/2}\medskip\\
\quad\leq2\left(  \mathbb{E}\left(  \sup\limits_{s\in\left[  t,t+h_{n}\right]
}\left\Vert
%TCIMACRO{\dint _{t}^{s}}%
%BeginExpansion
{\displaystyle\int_{t}^{s}}
%EndExpansion
\left(  B\left(  X_{r}\right)  -\nu\right)  dW_{r}\right\Vert _{V^{\ast}}%
^{2}\right)  \right)  ^{1/2}\left(  \mathbb{E}\left(  \left\Vert
%TCIMACRO{\dint _{t}^{t+h_{n}}}%
%BeginExpansion
{\displaystyle\int_{t}^{t+h_{n}}}
%EndExpansion
\left(  A\left(  X_{s}\right)  -\gamma\right)  ds\right\Vert _{V^{\ast}}%
^{2}\right)  \right)  ^{1/2}\medskip\\
\quad\leq2\mathcal{C}_{H,V^{\ast}}\left(  C\mathbb{E}\left(
%TCIMACRO{\dint _{t}^{t+h_{n}}}%
%BeginExpansion
{\displaystyle\int_{t}^{t+h_{n}}}
%EndExpansion
\left\Vert B\left(  X_{s}\right)  -\nu\right\Vert _{L^{2}\left(  U;H\right)
}^{2}ds\right)  \right)  ^{1/2}\left(  h_{n}\mathbb{E}\left(
%TCIMACRO{\dint _{t}^{t+h_{n}}}%
%BeginExpansion
{\displaystyle\int_{t}^{t+h_{n}}}
%EndExpansion
\left\Vert A\left(  X_{s}\right)  -\gamma\right\Vert _{V^{\ast}}^{2}ds\right)
\right)  ^{1/2}\medskip\\
\quad\leq2\mathcal{C}_{H,V^{\ast}}\sqrt{C}h_{n}^{\frac{1}{2}}\left(
h_{n}\mathbb{E}\left(  \sup\limits_{r\in\left[  t,t+h_{n}\right]  }\left\Vert
B\left(  X_{r}\right)  -\nu\right\Vert _{L^{2}\left(  U;H\right)  }%
^{2}\right)  \right)  ^{1/2}\left(  h_{n}\mathbb{E}\left(  \sup\limits_{r\in
\left[  t,t+h_{n}\right]  }\left\Vert A\left(  X_{r}\right)  -\gamma
\right\Vert _{V^{\ast}}^{2}\right)  \right)  ^{1/2}\medskip\\
\quad\leq Ch_{n}^{3/2}.
\end{array}
\]

We split $I_{6}$ into three parts and we deduce%
\[%
\begin{array}
[c]{ccl}%
I_{6} & \leq & 2h_{n}\left(  \mathbb{E}\left(  \left\Vert x-y\right\Vert
_{V^{\ast}}^{2}\right)  \right)  ^{1/2}\left(  \mathbb{E}\left\Vert
a_{n}\right\Vert _{V^{\ast}}^{2}\right)  ^{1/2}\smallskip\\
&  & +h_{n}\left(  \mathbb{E}\left(  \left\Vert
%TCIMACRO{\dint _{t}^{t+h_{n}}}%
%BeginExpansion
{\displaystyle\int_{t}^{t+h_{n}}}
%EndExpansion
\left(  A\left(  X_{s}\right)  -\gamma\right)  ds\right\Vert _{V^{\ast}}%
^{2}\right)  \right)  ^{1/2}\left(  \mathbb{E}\left(  \left\Vert \alpha
_{n}\right\Vert _{V^{\ast}}^{2}\right)  \right)  ^{1/2}\smallskip\\
&  & +h_{n}\left(  \mathbb{E}\left(  \left\Vert
%TCIMACRO{\dint _{t}^{t+h_{n}}}%
%BeginExpansion
{\displaystyle\int_{t}^{t+h_{n}}}
%EndExpansion
\left(  B\left(  X_{s}\right)  -\nu\right)  dW_{s}\right\Vert _{V^{\ast}}%
^{2}\right)  \right)  ^{1/2}\left(  \mathbb{E}\left(  \left\Vert \alpha
_{n}\right\Vert _{V^{\ast}}\right)  \right)  ^{1/2}.
\end{array}
\]

We treat $I_{7}$ in the same manner and we get%
\[%
\begin{array}
[c]{l}%
\mathbb{E}\left(  -2\left\langle x-y+%
%TCIMACRO{\dint _{t}^{t+h_{n}}}%
%BeginExpansion
{\displaystyle\int_{t}^{t+h_{n}}}
%EndExpansion
\left(  A\left(  X_{s}\right)  -\gamma\right)  ds+%
%TCIMACRO{\dint _{t}^{t+h_{n}}}%
%BeginExpansion
{\displaystyle\int_{t}^{t+h_{n}}}
%EndExpansion
\left(  B\left(  X_{s}\right)  -\nu\right)  dW_{s},h_{n}^{\frac{1}{2}+\delta
}\beta_{n}\right\rangle _{V^{\ast}}\right)  \medskip\\
\quad\quad=-2\mathbb{E}\left(  \left\langle x-y,h_{n}^{\frac{1}{2}+\delta
}\beta_{n}\right\rangle _{V^{\ast}}\right)  +\mathbb{E}\left(  \left\langle
%TCIMACRO{\dint _{t}^{t+h_{n}}}%
%BeginExpansion
{\displaystyle\int_{t}^{t+h_{n}}}
%EndExpansion
\left(  A\left(  X_{s}\right)  -\gamma\right)  ds,h_{n}^{\frac{1}{2}+\delta
}\beta_{n}\right\rangle _{V^{\ast}}\right)  \medskip\\
\quad\quad+\mathbb{E}\left(  \left\langle
%TCIMACRO{\dint _{t}^{t+h_{n}}}%
%BeginExpansion
{\displaystyle\int_{t}^{t+h_{n}}}
%EndExpansion
\left(  B\left(  X_{s}\right)  -\nu\right)  dW_{s},h_{n}^{\frac{1}{2}+\delta
}\beta_{n}\right\rangle _{V^{\ast}}\right)
\end{array}
\]
First, the expectation of the first term is equal with zero because $\beta
_{n}$ is independent of $x-y$ and $\mathbb{E}\left(  \beta_{n}\right)  =0$,
according to (\ref{ConditiiAlfaBeta}-$\left(  iii\right)  \&\left(  iv\right)
$). For the other two, we proceed as follow. Using first the Cauchy-Schwarz
inequality, it follows that%
\[%
\begin{array}
[c]{l}%
\mathbb{E}\left(  \left\langle
%TCIMACRO{\dint _{t}^{t+h_{n}}}%
%BeginExpansion
{\displaystyle\int_{t}^{t+h_{n}}}
%EndExpansion
\left(  A\left(  X_{s}\right)  -\gamma\right)  ds,h_{n}^{\frac{1}{2}+\delta
}\beta_{n}\right\rangle _{V^{\ast}}\right)  \medskip\\
\quad\quad\leq\left(  \mathbb{E}\left(  \left\Vert
%TCIMACRO{\dint _{t}^{t+h_{n}}}%
%BeginExpansion
{\displaystyle\int_{t}^{t+h_{n}}}
%EndExpansion
\left(  A\left(  X_{s}\right)  -\gamma\right)  ds\right\Vert _{V^{\ast}}%
^{2}\right)  \right)  ^{1/2}\left(  \mathbb{E}\left(  \left\Vert h_{n}%
^{\frac{1}{2}+\delta}\beta_{n}\right\Vert _{V^{\ast}}^{2}\right)  \right)
^{1/2}\medskip\\
\quad\quad\leq h_{n}^{1/2+\delta}\left(  h_{n}\mathbb{E}\left(
%TCIMACRO{\dint _{t}^{t+h_{n}}}%
%BeginExpansion
{\displaystyle\int_{t}^{t+h_{n}}}
%EndExpansion
\left\Vert A\left(  X_{s}\right)  -\gamma\right\Vert _{V^{\ast}}^{2}ds\right)
\right)  ^{1/2}\left(  \mathbb{E}\left(  \left\Vert \beta_{n}\right\Vert
_{V^{\ast}}^{2}\right)  \right)  ^{1/2}\medskip\\
\quad\quad=h_{n}^{1+\delta}\left(  \mathbb{E}\left(
%TCIMACRO{\dint _{t}^{t+h_{n}}}%
%BeginExpansion
{\displaystyle\int_{t}^{t+h_{n}}}
%EndExpansion
\left\Vert A\left(  X_{s}\right)  -\gamma\right\Vert _{V^{\ast}}^{2}ds\right)
\right)  ^{1/2}\left(  \mathbb{E}\left(  \left\Vert \beta_{n}\right\Vert
_{V^{\ast}}^{2}\right)  \right)  ^{1/2}.
\end{array}
\]
For the third term from $I_{7},$ we have%
\[%
\begin{array}
[c]{l}%
\mathbb{E}\left(  \left\langle
%TCIMACRO{\dint _{t}^{t+h_{n}}}%
%BeginExpansion
{\displaystyle\int_{t}^{t+h_{n}}}
%EndExpansion
\left(  B\left(  X_{s}\right)  -\nu\right)  dW_{s},h_{n}^{\frac{1}{2}+\delta
}\beta_{n}\rangle_{V^{\ast}}\right\rangle \right)  \medskip\\
\quad\quad\leq\left(  \mathbb{E}\left(  \sup\limits_{s\in\left[
t,t+h_{n}\right]  }\left\Vert
%TCIMACRO{\dint _{t}^{s}}%
%BeginExpansion
{\displaystyle\int_{t}^{s}}
%EndExpansion
\left(  B\left(  X_{r}\right)  -\nu\right)  dW_{r}\right\Vert _{V^{\ast}}%
^{2}\right)  \right)  ^{1/2}\left(  \mathbb{E}\left\Vert h_{n}^{\frac{1}%
{2}+\delta}\beta_{n}\right\Vert _{V^{\ast}}^{2}\right)  ^{1/2}\medskip\\
\quad\quad\leq\mathcal{C}_{H,V^{\ast}}h_{n}^{\frac{1}{2}+\delta}\left(
C\mathbb{E}\left(
%TCIMACRO{\dint _{t}^{t+h_{n}}}%
%BeginExpansion
{\displaystyle\int_{t}^{t+h_{n}}}
%EndExpansion
\left\Vert B\left(  X_{s}\right)  -\nu\right\Vert _{L_{2}\left(  U;H\right)
}^{2}ds\right)  \right)  ^{1/2}\left(  \mathbb{E}\left\Vert \beta
_{n}\right\Vert _{V^{\ast}}^{2}\right)  ^{1/2}\medskip\\
\quad\quad\leq\mathcal{C}_{H,V^{\ast}}h_{n}^{\frac{1}{2}+\delta}\left(
C\mathbb{E}\left(
%TCIMACRO{\dint _{t}^{t+h_{n}}}%
%BeginExpansion
{\displaystyle\int_{t}^{t+h_{n}}}
%EndExpansion
\sup\limits_{r\in\left[  t,t+h_{n}\right]  }\left\Vert B\left(  X_{r}\right)
-\nu\right\Vert _{L_{2}\left(  U;H\right)  }^{2}ds\right)  \right)
^{1/2}\left(  \mathbb{E}\left\Vert \beta_{n}\right\Vert _{V^{\ast}}%
^{2}\right)  ^{1/2}\medskip\\
\quad\quad\leq\mathcal{C}_{H,V^{\ast}}h_{n}^{\frac{1}{2}+\delta}\left(
Ch_{n}\mathbb{E}\left(  \sup\limits_{r\in\left[  t,t+h_{n}\right]  }\left\Vert
B\left(  X_{r}\right)  -\nu\right\Vert _{L_{2}\left(  U;H\right)  }%
^{2}\right)  \right)  ^{1/2}\left(  \mathbb{E}\left\Vert \beta_{n}\right\Vert
_{V^{\ast}}^{2}\right)  ^{1/2}\medskip\\
\quad\quad\leq Ch_{n}^{1+\delta}\mathbb{E}\left(  \sup\limits_{r\in\left[
t,t+h_{n}\right]  }\left\Vert B\left(  X_{r}\right)  -\nu\right\Vert
_{L_{2}\left(  U;H\right)  }^{2}\right)  ^{1/2}\left(  \mathbb{E}\left\Vert
\beta_{n}\right\Vert _{V^{\ast}}^{2}\right)  ^{1/2}.
\end{array}
\]
Finally, for $I_{8}$,%
\[
\mathbb{E}\left(  \left\Vert h_{n}\alpha_{n}+h_{n}^{\frac{1}{2}+\delta}%
\beta_{n}\right\Vert _{V^{\ast}}^{2}\right)  \leq2h_{n}^{2}\mathbb{E}\left(
\left\Vert \alpha_{n}\right\Vert _{V^{\ast}}^{2}\right)  +2h_{n}^{1+2\delta
}\mathbb{E}\left(  \left\Vert \beta_{n}\right\Vert _{V^{\ast}}^{2}\right)
\]
Putting now together into (\ref{delta_distance}) all the above estimates
obtained for $I_{1}$ to $I_{8}$ and dividing by $h_{n}$ yields%
\begin{equation}%
\begin{array}
[c]{l}%
\liminf\limits_{h_{n}\rightarrow0+}\dfrac{1}{h_{n}}\left(  \mathbb{E}\left(
d_{K}^{2}\left(  X_{t+h_{n}}\right)  \right)  -\mathbb{E}\left(  d_{K}%
^{2}\left(  X_{t}\right)  \right)  \right)  \medskip\\
\quad\quad\quad\quad\leq2\mathbb{E}\left(  \langle X_{t}-y,A\left(
X_{t}\right)  -\gamma\rangle_{V^{\ast}}\right)  +3\mathbb{E}\left(  \left\Vert
B\left(  X_{t}\right)  -\nu\right\Vert _{L^{2}\left(  U;H\right)  }%
^{2}\right)  ,
\end{array}
\label{estimates for increments of expectation of d2}%
\end{equation}
for any $\mathcal{F}_{t}-$measurable $y\left(  t\right)  \in\Pi_{K}\left(
X_{t}\right)  \subset V$ and any $\mathcal{F}_{t}-$measurable vector $\left(
\gamma\left(  t\right)  ,\nu\left(  t\right)  \right)  \in\mathcal{T}%
_{K}\left(  t,y\left(  t\right)  \right)  .$ The above inequality can also be
considered by replacing $h_{n}$ and its convergence to zero by a continuous
parameter $h$ converging to zero.

Since there exists a selection $y\left(  t\right)  $ of $\Pi_{K}\left(
X_{t}\right)  $, we can take
\[
\gamma\left(  t\right)  :=A\left(  X_{t}\right)  \quad\quad\text{and}%
\quad\quad\nu\left(  t\right)  :=B\left(  X_{t}\right)  .
\]
We obtain, given the notation $\phi\left(  t\right)  :=\mathbb{E}\left(
d_{K}^{2}\left(  X_{t}\right)  \right)  $, that the contingent epiderivate
satisfy:%
\[
D_{\uparrow\phi\left(  t\right)  \left(  1\right)  }=\liminf_{h\rightarrow
0+}\frac{\left(  \mathbb{E}\left(  d_{K}^{2}\left(  X_{t+h}\right)  \right)
-\mathbb{E}\left(  d_{K}^{2}\left(  X_{t}\right)  \right)  \right)  }{h}%
\leq0.
\]
Based on the Lyapunov function theory presented in Aubin \cite{Aubin:90}, this
property of the contingent epiderivative implies that
\[
\phi\left(  t\right)  \leq\phi\left(  t_{0}\right)  =0,\quad\quad\text{for all
}t\geq t_{0}.
\]
Since we also have $\phi\left(  t\right)  \geq0$, it follows that, for every
$t\in\left[  0,T\right]  ,$ we have%
\[
\mathbb{E}\left(  d_{K}^{2}\left(  X_{t}\right)  \right)  =\int_{\Omega
}d_{K_{\omega}}^{2}\left(  X_{\omega}\left(  t\right)  \right)  dP\left(
w\right)  =0,
\]
which implies, since the integrand is nonnegative, that $d_{K_{\omega}}%
^{2}\left(  X_{\omega}\left(  t\right)  \right)  =0,$ $\mathbb{P-}a.s.$, i.e.,
the stochastic process $X$ is viable in $K.$

The theorems' proof is now completed.\hfill
\end{proof}

\begin{remark}
Alternatively, one can adopt a second approach for the last assertion from the
viability result involving $K$ and $\left(  A,B\right)  $. In order to
consider our study complete, we present it with rigorous detailed arguments.
Since Theorem \ref{main result} proves, in the first two implications, some
necessary conditions for viability, we continue to investigate, in depth, the
last one. More precisely, we focus on the more technical $^{\prime\prime
}(iii)\Rightarrow(i)^{\prime\prime}$, which establish, in fact, some
sufficient conditions for the viability to hold.
\end{remark}

\begin{theorem}
Assume that $K$ is a closed subset of\ $V$ and the maps $A:V\rightarrow
V^{\ast}$ and $B:V\rightarrow L_{2}\left(  U,H\right)  \ $are uniformly
continuous and are satisfying $\left(  H_{1}\right)  $ to $\left(
H_{7}\right)  .$ Moreover, suppose that, for every $t\in\left[  0,T\right]  $
and for every $\mathcal{F}_{t}$ - random variable $x\in L^{2}\left(
\Omega;V\right)  $ from $K$ we have that%
\[
\left(  A\left(  x\right)  ,B\left(  x\right)  \right)  \in\mathcal{T}%
_{K}\left(  t,x\right)  .
\]
Then, the following assertions hold:

$\left(  I\right)  \quad$For any fixed $\epsilon>0$, the set $S_{\epsilon
}\left(  X_{0}\right)  $ of "close enough" stochastic processes $X_{\epsilon
}=\left(  X_{\epsilon,t}\right)  _{t\in\left[  0,T\right]  }$ satisfying
$X_{\epsilon}\left(  0\right)  =X_{0}$ and%
\begin{equation}
\left\{
\begin{array}
[c]{cl}%
a)\quad & \forall t\in\left[  0,T\right]  ,\quad\mathbb{E}\left(  d_{V^{\ast}%
}^{2}\left(  X_{\epsilon,t},\mathcal{K}\right)  \right)  \leq\epsilon
^{2}\medskip\\
b)\quad & \forall t\in\left[  0,T\right]  ,\quad\mathbb{E}\left(  \left\Vert
X_{\epsilon,t}-X_{0}-%
%TCIMACRO{\dint _{0}^{t}}%
%BeginExpansion
{\displaystyle\int_{0}^{t}}
%EndExpansion
A\left(  X_{\epsilon,s}\right)  ds-%
%TCIMACRO{\dint _{0}^{t}}%
%BeginExpansion
{\displaystyle\int_{0}^{t}}
%EndExpansion
B\left(  X_{\epsilon,s}\right)  dW_{s}\right\Vert _{V^{\ast}}^{2}\right)
\leq\epsilon^{2}%
\end{array}
\right.  \label{LemaSuf}%
\end{equation}
is not empty.$\smallskip$

$\left(  II\right)  \quad$Passing to limit, as $\epsilon\rightarrow0$, there
exists a unique stochastic process $X$ on $\left[  0,T\right]  $, which is
$\mathcal{K-}$viable.
\end{theorem}

\begin{proof}
$\left(  I\right)  \quad$For an arbitrary fixed $\epsilon>0$, one can select
an $\eta:=\eta_{\epsilon}\in(0,\epsilon]$ such that%
\begin{equation}
\theta_{A\wedge B}^{conc}\left(  \eta^{2}\right)  :=\operatorname{Concave}%
\left\{  \max\left\{  \theta_{A},\theta_{B}\right\}  \right\}  \left(
\eta^{2}\right)  \leq\frac{\epsilon^{2}}{2T+6\mathcal{C}_{H,V^{\ast}}^{2}},
\label{eta chose}%
\end{equation}
where $\theta_{A},\theta_{B}$\footnote{Given a function $f,$ set
$\theta:\left[  0,T\right]  \rightarrow\mathbb{R}_{+}$, $\theta_{f}\left(
\eta\right)  :=\sup_{\left\Vert x-y\right\Vert ^{2}\leq\eta}\left\Vert
f\left(  x\right)  -f\left(  y\right)  \right\Vert ^{2}$ as being the concave
modulus of continuity of $f.$ Then $\theta$ is a non decreasing, subadditive
and concave function which dominates the variation of $f.$ Usually, it is
chosen as the smallest concave function with the above properties and it is
the concave span of the classical modulus of continuity of the given function
$f$.} are the concave modulus of continuity for $A$ and $B.$ Since the maximum
of two concave function is not necessary concave, in order to maximize both
modulus with a concave function, one can chose the concave envelope of their
maximum. It can also be controlled to small values since the involved two
modulus have this property.

We denote by $G_{\epsilon}\left(  X_{0}\right)  $ the set of $\epsilon
-$dependent pairs $\left(  T_{X},X\left(  \cdot\right)  \right)  :=\left(
T_{X_{\epsilon}},X_{\epsilon}\left(  \cdot\right)  \right)  $ where $T_{X}%
\in\left[  0,T\right]  $ and $X\left(  \cdot\right)  $ is a stochastic process
that satisfies $X\left(  0\right)  =X_{0}$ and
\begin{equation}
\left\{
\begin{array}
[c]{ll}%
a)\quad & \mathbb{E}\left(  d_{V^{\ast}}^{2}\left(  X_{T_{X}},\mathcal{K}%
\right)  \right)  \leq\eta^{2}\dfrac{e^{T_{X}}-1}{e^{T}-1}\leq\eta
^{2}~,\medskip\\
b)\quad & \forall t\in\left[  0,T_{X}\right]  ,\quad\mathbb{E}\left(
d_{V^{\ast}}^{2}\left(  X_{t},\mathcal{K}\right)  \right)  \leq\eta
^{2}~,\medskip\\
c)\quad & \forall t\in\left[  0,T_{X}\right]  ,\quad\mathbb{E}\left(
\left\Vert X_{t}-X_{0}-%
%TCIMACRO{\dint _{0}^{t}}%
%BeginExpansion
{\displaystyle\int_{0}^{t}}
%EndExpansion
A\left(  X_{s}\right)  ds-%
%TCIMACRO{\dint _{0}^{t}}%
%BeginExpansion
{\displaystyle\int_{0}^{t}}
%EndExpansion
B\left(  X_{s}\right)  dW_{s}\right\Vert _{V^{\ast}}^{2}\right)  \leq
\epsilon^{2},
\end{array}
\right.  \label{Set G_eps definition}%
\end{equation}
where, we denoted, as usual, $d_{V^{\ast}}^{2}\left(  X,\mathcal{K}\right)
=\inf_{u\in\mathcal{K}}\left\Vert X-u\right\Vert _{V^{\ast}}^{2}.$

Obviously, $G_{\epsilon}\left(  X_{0}\right)  \neq\emptyset$, since $\left(
0,X_{0}\right)  \in G_{\epsilon}\left(  X_{0}\right)  $. On the other hand,
$G_{\epsilon}\left(  X_{0}\right)  $ is an inductive set with respect to the
order relation%
\[
\left(  T_{X_{1}},X_{1}\left(  \cdot\right)  \right)  \preceq\left(  \left(
T_{X_{2}},X_{2}\left(  \cdot\right)  \right)  \right)
\]
defined by%
\[
T_{X_{1}}\leq T_{X_{2}}\quad\text{and}\quad X_{2}\left(  \cdot\right)
|_{\left[  0,T_{X_{1}}\right]  }=X_{1}\left(  \cdot\right)  .
\]
Consequently, by Zorn's Lemma, it follows that there exists a maximal element
$\left(  T_{X},X\left(  \cdot\right)  \right)  \in G_{\epsilon}\left(
X_{0}\right)  $ and we can suppose, for the sake of contradiction, that
$T_{X}$ is not the maturity date of the state process, i.e. $T_{X}<T.$ Our
goal is to extend $X\left(  \cdot\right)  $ to a stochastic process $\hat
{X}\left(  \cdot\right)  $ on an interval $\left[  T_{X},T_{\hat{X}}\right]  $
where $T_{\hat{X}}>T_{X}$ .$\smallskip$

As a random variable, the projection map $\Pi_{K_{\omega}}\left(  X_{\omega
}\left(  T_{X}\right)  \right)  $ is $\mathcal{F}_{T_{X}}$-measurable, since
$K_{\omega}$ and $X_{\omega}\left(  T_{X}\right)  $ are $\mathcal{F}_{T_{X}}$
measurable random set/variable. Therefore, there exists an $\mathcal{F}%
_{T_{X}}$ measurable selection%
\[
y_{\omega}\in\Pi_{K_{\omega}}\left(  X_{\omega,T_{X}}\right)  ,
\]
as being a projection of the random variable $X_{T_{X}}$ onto the random
set-valued variable $K.$ Denote $x=X_{T_{X}}$ and we may choose a projection
$y\in\Pi_{K}\left(  x\right)  .$ We know that%
\[
\left(  A\left(  y\right)  ,B\left(  y\right)  \right)  \in\mathcal{T}%
_{K}\left(  T_{X},y\right)  .
\]
This means that there exist sequences $t_{n}>0$ converging to $0$ and
$\mathcal{F}_{t}$-measurable $\alpha_{n}$ and $\beta_{n}$ satisfying the
assumptions (\ref{ConditiiAlfaBeta}) and, for every $n\geq0,$%
\[
y+B\left(  y\right)  \left(  W_{T_{X}+t_{n}}-W_{T_{X}}\right)  +A\left(
y\right)  t_{n}+t_{n}\alpha_{n}+t_{n}^{\frac{1}{2}+\delta}\beta_{n}\in
K,\quad\mathbb{P-}a.s.
\]
Since $t_{n}$ converges to zero it is possible to choose a sufficiently large
positive integer $N$ such that $t_{N}\leq\phi$, with%
\[
\phi:=\min\left(  \eta,\frac{\eta^{2}\left(  1-\dfrac{e^{T_{X}}-1}{e^{T}%
-1}\right)  }{\eta^{2}+\left(  1+T+3\mathcal{C}_{H,V^{\ast}}^{2}\right)
c^{2}}\right)  >0,
\]
with $\eta$ given by (\ref{eta chose}) and%
\[
c^{2}:=\max\left(  \mathbb{E}\left(  \left\Vert A\left(  y\right)  \right\Vert
_{V^{\ast}}^{2}\right)  ,\mathbb{E}\left(  \left\Vert B\left(  y\right)
\right\Vert _{L_{2}\left(  U;H\right)  }^{2}\right)  \right)  <+\infty.
\]
Define also%
\[
\rho:=\frac{\eta}{\sqrt{\left(  e^{T}-1\right)  \left(  3+2h\right)  }}>0.
\]
where $h:=t_{N}$.

Set $T_{\hat{X}}:=T_{X}+h>T_{X}~$. We define $\hat{X}_{t}$ on the extension
time interval $\left[  T_{X},T_{\hat{X}}\right]  $ as follows%
\[
\hat{X}_{t}:=x+\left(  t-T_{X}\right)  A\left(  y\right)  +B\left(  y\right)
\left(  W_{t}-W_{T_{X}}\right)  ,
\]
or, under integral form,%
\[
\hat{X}_{t}=x+\int_{T_{X}}^{t}A\left(  y\right)  ds+\int_{T_{X}}^{t}B\left(
y\right)  dW_{s}.
\]

Hence, one can easily see that%
\[%
\begin{array}
[c]{c}%
d_{V^{\ast}}^{2}\left(  \hat{X}_{T_{\hat{X}}},\mathcal{K}\right)  -d_{V^{\ast
}}^{2}\left(  \hat{X}_{T_{X}},\mathcal{K}\right)  \leq\left\Vert
x-y-h\alpha_{n}-h^{\frac{1}{2}+\delta}\beta_{n}\right\Vert _{V^{\ast}}%
^{2}-\left\Vert x-y\right\Vert _{V^{\ast}}^{2}\medskip\\
=\left\Vert h\alpha_{n}+h^{\frac{1}{2}+\delta}\beta_{n}\right\Vert _{V^{\ast}%
}^{2}-2\langle x-y,h\alpha_{n}\rangle_{V^{\ast}}-2\left\langle x-y,h^{\frac
{1}{2}+\delta}\beta_{n}\right\rangle _{V^{\ast}}.
\end{array}
\]
Applying the expectation in both sides of this inequality, we provide
estimates for each term on the right hand-side. First, we observe that%
\[
\mathbb{E}\left(  \left\Vert h\alpha_{n}+h^{\frac{1}{2}+\delta}\beta
_{n}\right\Vert _{V^{\ast}}^{2}\right)  \leq2h\left(  h\mathbb{E}\left(
\left\Vert \alpha_{n}\right\Vert _{V^{\ast}}^{2}\right)  +h^{2\delta
}\mathbb{E}\left(  \left\Vert \beta_{n}\right\Vert _{V^{\ast}}^{2}\right)
\right)
\]
and%
\[
\mathbb{E}\left(  \langle x-y,h\alpha_{n}\rangle_{V^{\ast}}\right)  \leq
h\left(  \mathbb{E}\left(  \left\Vert x-y\right\Vert _{V^{\ast}}^{2}\right)
\right)  ^{1/2}\left(  \mathbb{E}\left\Vert \alpha_{n}\right\Vert _{V^{\ast}%
}^{2}\right)  ^{1/2}%
\]
We also have, as before, $\mathbb{E}(\langle x-y,h^{\delta+1/2}\beta
_{n}\rangle_{V^{\ast}})=0$, due to the properties of $\beta_{n}$ from
Definition \ref{defT}.

By (\ref{Set G_eps definition}$-\left(  a\right)  $),%
\[%
\begin{array}
[c]{l}%
\mathbb{E}\left(  d_{V^{\ast}}^{2}(\hat{X}_{T_{\hat{X}}},\mathcal{K})\right)
=\mathbb{E}\left(  d_{V^{\ast}}^{2}(\hat{X}_{T_{X}+h},\mathcal{K})\right)
\bigskip\\
\quad\quad\leq\mathbb{E}\left(  d_{V^{\ast}}^{2}(\hat{X}_{T_{X}}%
,\mathcal{K)}\right)  +2h\left(  \mathbb{E}(\left\Vert x-y\right\Vert
_{V^{\ast}}^{2})\right)  ^{1/2}\left(  \mathbb{E}(\left\Vert \alpha
_{n}\right\Vert _{V^{\ast}}^{2})\right)  ^{1/2}\bigskip\\
\quad\quad\quad+2h\left(  h\mathbb{E}\left(  \left\Vert \alpha_{n}\right\Vert
_{V^{\ast}}^{2}\right)  \right)  +h^{2\delta}\mathbb{E}\left(  \left\Vert
\beta_{n}\right\Vert _{V^{\ast}}^{2}\right)  \bigskip\\
\quad\quad\leq\mathbb{E}\left(  d_{V^{\ast}}^{2}\left(  \hat{X}_{T_{X}%
},\mathcal{K}\right)  \right)  +h\left(  \mathbb{E}\left(  \left\Vert
x-y\right\Vert _{V^{\ast}}^{2}\right)  +\left(  1+2h\right)  \mathbb{E}\left(
\left\Vert \alpha_{n}\right\Vert _{V^{\ast}}^{2}\right)  +2\mathbb{E}\left(
\left\Vert \beta_{n}\right\Vert _{V^{\ast}}^{2}\right)  \right)  \bigskip\\
\quad\quad\overset{(\ref{Set G_eps definition}-\left(  a\right)  )}{\leq}%
\eta^{2}\dfrac{e^{T_{X}}-1}{e^{T}-1}+h\left(  \eta^{2}\dfrac{e^{T_{X}}%
-1}{e^{T}-1}+\left(  3+2h\right)  \rho^{2}\right)  \leq\eta^{2}\dfrac
{e^{T_{X}}-1}{e^{T}-1}+\dfrac{h(e^{T_{X}}-1+1)}{e^{T}-1}\eta^{2}\bigskip\\
\quad\quad\leq\eta^{2}\left(  \dfrac{e^{T_{X}}\left(  1+h)-1\right)  }%
{e^{T}-1}\right)  \leq\eta^{2}\dfrac{e^{T_{\hat{X}}}-1}{e^{T}-1}%
,\quad\text{since }1+h\leq e^{h}.
\end{array}
\]
It follows that $\hat{X}$ satisfies\ (\ref{Set G_eps definition}$-\left(
a\right)  $) for $T_{\hat{X}}.$

Furthermore, for proving (\ref{Set G_eps definition}$-\left(  b\right)  $),
observe that, for any $t\in\left[  T_{X},T_{\hat{X}}\right]  ,$%
\[
d_{V^{\ast}}^{2}\left(  \hat{X}_{t},\mathcal{K}\right)  \leq\left\Vert \hat
{X}_{t}-y\right\Vert _{V^{\ast}}^{2}.
\]
Hence%
\[%
\begin{array}
[c]{ccl}%
\left\Vert \hat{X}_{t}-y\right\Vert _{V^{\ast}}^{2} & = & \left\Vert
x-y+\left(  t-T_{X}\right)  A\left(  y\right)  +\left(  W_{t}-W_{T_{X}%
}\right)  B\left(  y\right)  \right\Vert _{V^{\ast}}^{2}\medskip\\
& = & d_{V^{\ast}}^{2}\left(  x,\mathcal{K}\right)  +2\left\langle x-y,\left(
t-T_{X}\right)  A\left(  y\right)  +B\left(  y\right)  \left(  W_{t}-W_{T_{X}%
}\right)  \right\rangle _{V^{\ast}}\medskip\\
& + & \left\Vert \left(  t-T_{X}\right)  A\left(  y\right)  +B\left(
y\right)  \left(  W_{t}-W_{T_{X}}\right)  \right\Vert _{V^{\ast}}^{2}~.
\end{array}
\]
By taking the expectations, we obtain%
\[%
\begin{array}
[c]{l}%
\mathbb{E}\left(  \left\Vert \hat{X}_{t}-y\right\Vert _{V^{\ast}}^{2}\right)
-\mathbb{E}\left(  d_{V^{\ast}}^{2}\left(  \hat{X}_{T_{X}},\mathcal{K}\right)
\right)  \medskip\\
\quad\quad=2\mathbb{E}\left(  \left\langle x-y,\left(  t-T_{X}\right)
A\left(  y\right)  \right\rangle _{V^{\ast}}\right)  +2\mathbb{E}\left(
\left\langle x-y,B\left(  y\right)  \left(  W_{t}-W_{T_{X}}\right)
\right\rangle _{V^{\ast}}\right)  \medskip\\
\quad\quad+\mathbb{E}\left(  \left\Vert \left(  t-T_{X}\right)  A\left(
y\right)  \right\Vert _{V^{\ast}}^{2}\right)  +\mathbb{E}\left(  \left\Vert
B\left(  y\right)  \left(  W_{t}-W_{T_{X}}\right)  \right\Vert _{V^{\ast}}%
^{2}\right)  \medskip\\
\quad\quad+2\mathbb{E}\left(  \left\langle \left(  t-T_{X}\right)  A\left(
y\right)  ,B\left(  y\right)  \left(  W_{t}-W_{T_{X}}\right)  \right\rangle
_{V^{\ast}}\right)  .
\end{array}
\]
Remark that%
\[
\mathbb{E}\left(  \left\langle x-y,\int_{0}^{t_{n}}B\left(  y\right)
dW_{s}\right\rangle _{V^{\ast}}\right)  =\mathbb{E}\left(  \int_{0}^{t_{n}%
}\left\langle x-y,B\left(  y\right)  dW_{s}\right\rangle _{V^{\ast}}\right)
.
\]
We also note that%
\[%
\begin{array}
[c]{l}%
\mathbb{E}\left(  \left\Vert B\left(  y\right)  \left(  W_{t}-W_{T_{X}%
}\right)  \right\Vert _{V^{\ast}}^{2}\right)  =\mathbb{E}\left(  \left\Vert
%TCIMACRO{\dint _{T_{X}}^{t}}%
%BeginExpansion
{\displaystyle\int_{T_{X}}^{t}}
%EndExpansion
B\left(  y\right)  dW_{s}\right\Vert _{V^{\ast}}^{2}\right)  \medskip\\
\quad\quad\leq\mathbb{E}\left(  \sup\limits_{s\in\left[  T_{X},t\right]
}\left\Vert
%TCIMACRO{\dint _{T_{X}}^{s}}%
%BeginExpansion
{\displaystyle\int_{T_{X}}^{s}}
%EndExpansion
B\left(  y\right)  dW_{s}\right\Vert _{V^{\ast}}^{2}\right)  \leq
3\mathcal{C}_{H,V^{\ast}}^{2}\mathbb{E}\left(
%TCIMACRO{\dint _{T_{X}}^{t}}%
%BeginExpansion
{\displaystyle\int_{T_{X}}^{t}}
%EndExpansion
\left\Vert B\left(  y\right)  \right\Vert _{L_{2}\left(  U;H)\right)  }%
^{2}ds\right)  \medskip\\
\quad\quad\leq3\mathcal{C}_{H,V^{\ast}}^{2}\left(  t-T_{X}\right)
\mathbb{E}\left(  \left\Vert B\left(  y\right)  \right\Vert _{L_{2}\left(
U;H\right)  }^{2}\right)  .
\end{array}
\]
Using the above estimates, we obtain%
\[%
\begin{array}
[c]{l}%
\mathbb{E}\left(  \left\Vert \hat{X}_{t}-y\right\Vert _{V^{\ast}}^{2}\right)
-\mathbb{E}\left(  d_{V^{\ast}}^{2}\left(  \hat{X}_{T_{X}},\mathcal{K}\right)
\right)  \leq\left(  t-T_{X}\right)  \left(  \mathbb{E}(\left\Vert
x-y\right\Vert _{V^{\ast}}^{2})+\mathbb{E(}\left\Vert A\left(  y\right)
\right\Vert _{V^{\ast}}^{2})\right)  \bigskip\\
\quad\quad+\left(  t-T_{X}\right)  ^{2}\mathbb{E}\left(  \left\Vert A\left(
y\right)  \right\Vert _{V^{\ast}}^{2}\right)  +3\mathcal{C}_{H,V^{\ast}}%
^{2}\left(  t-T_{X}\right)  \mathbb{E}\left(  \left\Vert B\left(  y\right)
\right\Vert _{L_{2}\left(  U;H\right)  }^{2}\right)  \bigskip\\
\quad\quad\leq\left(  t-T_{X}\right)  \left(  \mathbb{E}\left(  d_{V^{\ast}%
}^{2}\left(  \hat{X}_{T_{X}},\mathcal{K}\right)  \right)  +\left(  1+\left(
t-T_{X}\right)  \right)  \mathbb{E}\left(  \left\Vert A\left(  y\right)
\right\Vert _{V^{\ast}}^{2}\right)  +3\mathcal{C}_{H,V^{\ast}}^{2}%
\mathbb{E}\left(  \left\Vert B\left(  y\right)  \right\Vert _{L_{2}\left(
U;H\right)  }^{2}\right)  \right)  .
\end{array}
\]
Since $c^{2}=\max\left\{  \mathbb{E}(\left\Vert A\left(  y\right)  \right\Vert
_{V^{\ast}}^{2}),\mathbb{E}(\left\Vert B\left(  y\right)  \right\Vert
_{L_{2}\left(  U;H\right)  }^{2})\right\}  $ and%
\[
\phi\left(  \eta^{2}+\left(  1+T+3\mathcal{C}_{H,V^{\ast}}^{2}\right)
c^{2}\right)  \leq\eta^{2}\left(  1-\frac{e^{T_{X}}-1}{e^{T}-1}\right)  ,
\]
it follows that%
\[%
\begin{array}
[c]{ccl}%
\mathbb{E}\left(  \left\Vert \hat{X}_{t}-y\right\Vert _{V^{\ast}}^{2}\right)
& \leq & \eta^{2}\dfrac{e^{T_{X}}-1}{e^{T}-1}+\left(  t-T_{X}\right)  \left(
\eta^{2}\dfrac{e^{T_{X}}-1}{e^{T}-1}+\left(  1+T+3\mathcal{C}_{H,V^{\ast}}%
^{2}\right)  c^{2}\right)  \medskip\\
& \leq & \eta^{2}\dfrac{e^{T_{X}}-1}{e^{T}-1}+\phi\left(  \eta^{2}+\left(
1+T+3\mathcal{C}_{H,V^{\ast}}^{2}\right)  c^{2}\right)  \medskip\\
& \leq & \eta^{2}.
\end{array}
\]
We conclude that $\hat{X}_{\cdot}$ satisfies (\ref{Set G_eps definition}%
$-\left(  b\right)  $):%
\begin{equation}
\mathbb{E}\left(  d_{V^{\ast}}^{2}\left(  \hat{X}_{t},\mathcal{K}\right)
\right)  \leq\mathbb{E}\left(  \left\Vert \hat{X}_{t}-y\right\Vert _{V^{\ast}%
}^{2}\right)  \leq\eta^{2}. \label{important boundedness}%
\end{equation}

In order to produce (\ref{Set G_eps definition}$-\left(  c\right)  $), we
start with%
\[%
\begin{array}
[c]{l}%
\mathbb{E}\left(  \left\Vert \hat{X}_{t}-x-%
%TCIMACRO{\dint _{T_{X}}^{t}}%
%BeginExpansion
{\displaystyle\int_{T_{X}}^{t}}
%EndExpansion
A(\hat{X}_{s})ds-%
%TCIMACRO{\dint _{T_{X}}^{t}}%
%BeginExpansion
{\displaystyle\int_{T_{X}}^{t}}
%EndExpansion
B(\hat{X}_{s})dW_{s}\right\Vert _{V^{\ast}}^{2}\right)  \medskip\\
\quad=\mathbb{E}\left(  \left\Vert
%TCIMACRO{\dint _{T_{X}}^{t}}%
%BeginExpansion
{\displaystyle\int_{T_{X}}^{t}}
%EndExpansion
\left(  A\left(  y\right)  -A(\hat{X}_{s})\right)  ds+%
%TCIMACRO{\dint _{T_{X}}^{t}}%
%BeginExpansion
{\displaystyle\int_{T_{X}}^{t}}
%EndExpansion
\left(  B\left(  y\right)  -B(\hat{X}_{s})\right)  dW_{s}\right\Vert
_{V^{\ast}}^{2}\right)  \medskip\\
\quad\leq2\left(  t-T_{X}\right)  \mathbb{E}\left(
%TCIMACRO{\dint _{T_{X}}^{t}}%
%BeginExpansion
{\displaystyle\int_{T_{X}}^{t}}
%EndExpansion
\left\Vert A\left(  y\right)  -A(\hat{X}_{s})\right\Vert _{V^{\ast}}%
^{2}ds\right)  +2\mathbb{E}\left(  \left\Vert
%TCIMACRO{\dint _{T_{X}}^{t}}%
%BeginExpansion
{\displaystyle\int_{T_{X}}^{t}}
%EndExpansion
\left(  B\left(  y\right)  -B(\hat{X}_{s})\right)  dW_{s}\right\Vert
_{V^{\ast}}^{2}\right)
\end{array}
\]
By the Burkholder--Davis--Gundy inequality it follows%
\begin{align*}
&  \mathbb{E}\left(  \left\Vert \hat{X}_{t}-x-\int_{T_{X}}^{t}A(\hat{X}%
_{s})ds-\int_{T_{X}}^{t}B(\hat{X}_{s})dW_{s}\right\Vert _{V^{\ast}}^{2}\right)
\\
&  \leq2T\mathbb{E}\left(  \int_{T_{X}}^{t}\left\Vert A\left(  y\right)
-A(\hat{X}_{s})\right\Vert _{V^{\ast}}^{2}ds\right)  +2\mathbb{E}\left(
\sup_{s\in\left[  T_{X},t\right]  }\left\Vert \int_{T_{X}}^{s}(B\left(
y\right)  -B(\hat{X}_{s}))dW_{s}\right\Vert _{V^{\ast}}^{2}\right) \\
&  \leq2T\mathbb{E}\left(  \int_{T_{X}}^{t}\left\Vert A\left(  y\right)
-A(\hat{X}_{s})\right\Vert _{V^{\ast}}^{2}ds\right)  +6\mathcal{C}_{H,V^{\ast
}}^{2}\mathbb{E}\left(
%TCIMACRO{\dint _{T_{X}}^{t}}%
%BeginExpansion
{\displaystyle\int_{T_{X}}^{t}}
%EndExpansion
\left\Vert B\left(  y\right)  -B(\hat{X}_{s})\right\Vert _{L_{2}\left(
U;H\right)  }^{2}ds\right)  .
\end{align*}
Using the concavity of the "joint modulus of continuity" $\theta_{A\wedge
B}^{conc}$ of the functions $A$ and $B$ and (\ref{eta chose}) we obtain%
\[%
\begin{array}
[c]{l}%
\mathbb{E}\left(  \left\Vert \hat{X}_{t}-x-%
%TCIMACRO{\dint _{T_{X}}^{t}}%
%BeginExpansion
{\displaystyle\int_{T_{X}}^{t}}
%EndExpansion
A(\hat{X}_{s})ds-%
%TCIMACRO{\dint _{T_{X}}^{t}}%
%BeginExpansion
{\displaystyle\int_{T_{X}}^{t}}
%EndExpansion
B(\hat{X}_{s})dW_{s}\right\Vert _{V^{\ast}}^{2}\right)  \medskip\\
\quad\leq2T\left(  \mathbb{E}\left(
%TCIMACRO{\dint _{T_{X}}^{t}}%
%BeginExpansion
{\displaystyle\int_{T_{X}}^{t}}
%EndExpansion
\theta_{A}\left(  \left\Vert y-\hat{X}_{s}\right\Vert _{V^{\ast}}^{2}\right)
ds\right)  \right)  +6\mathcal{C}_{H,V^{\ast}}^{2}\left(  \mathbb{E}\left(
%TCIMACRO{\dint _{T_{X}}^{t}}%
%BeginExpansion
{\displaystyle\int_{T_{X}}^{t}}
%EndExpansion
\theta_{B}\left(  \left\Vert y-\hat{X}_{s}\right\Vert _{V^{\ast}}^{2}\right)
ds\right)  \right)  \medskip\\
\quad\leq\left(  2T+6\mathcal{C}_{H,V^{\ast}}^{2}\right)  \mathbb{E}\left(
%TCIMACRO{\dint _{T_{X}}^{t}}%
%BeginExpansion
{\displaystyle\int_{T_{X}}^{t}}
%EndExpansion
\theta_{A\wedge B}^{conc}\left(  \left\Vert y-\hat{X}_{s}\right\Vert
_{V^{\ast}}^{2}\right)  ds\right)  \leq\left(  2T+6\mathcal{C}_{H,V^{\ast}%
}^{2}\right)  \theta_{A\wedge B}^{conc}\left(
%TCIMACRO{\dint _{T_{X}}^{t}}%
%BeginExpansion
{\displaystyle\int_{T_{X}}^{t}}
%EndExpansion
\mathbb{E}\left(  \left\Vert y-\hat{X}_{s}\right\Vert _{V^{\ast}}^{2}\right)
ds\right)  \medskip\\
\quad\leq\left(  2T+6\mathcal{C}_{H,V^{\ast}}^{2}\right)  \theta_{A\wedge
B}^{conc}\left(  \eta^{2}\right)  \leq\epsilon^{2},\quad\text{due to
(\ref{eta chose}).}%
\end{array}
\]
This implies that property (\ref{Set G_eps definition}$-\left(  c\right)  $)
is satisfied.

As consequence, we extended $X$ to a process $\hat{X}$ on the time interval
$\left[  T_{X},T_{\hat{X}}\right]  $, which is contradicts the maximality of
$\left(  T_{X},X\right)  $. The first part of the Proposition is now
completed.\medskip

$\left(  II\right)  \quad$One can rewrite (\ref{LemaSuf}$-\left(  b\right)
$), for every fixed $\epsilon>0$, as:%
\[
X_{\epsilon,t}\left(  t\right)  =X_{0}+\int_{0}^{t}A\left(  X_{\epsilon
,s}\right)  ds+\int_{0}^{t}B\left(  X_{\epsilon,s}\right)  dW_{s}%
+\zeta_{\epsilon,t}%
\]
where%
\[
\sup_{t\in\lbrack0,T]}\mathbb{E}\left(  \left\Vert \zeta_{\epsilon
,t}\right\Vert _{V^{\ast}}^{2}\right)  \leq\epsilon^{2}.
\]
We can advance now by showing that, making $\epsilon$ small enough, we derive
a unique solution for the initial equation (\ref{GeneralFormEq.}), which is
$\mathcal{K}-$viable. Usually, this can be achieved by proving that the
approximating process forms a Cauchy sequence. However, we should mention
that, according to $\left(  I\right)  ,$ for every $\epsilon,,\eta>0,$
$S_{\epsilon}\left(  X_{0}\right)  $ and $S_{\eta}\left(  X_{0}\right)  $ are
nonempty sets, with their cardinals, possible, greater than $1.$ This means
that we do not have, necessary, the same explicit formula, obtained by a
penalizing argument, for each element of the approximating sequence of
solutions. But, given the structural form of each set from the family of
"Cauchy sets", we can select an element from each one (for instance, the
maximal one constructed in the previous part of this Proposition) to play its
role in proving that we have obtained, in fact a Cauchy approximating
sequence, which leads to a unique limit process.$\smallskip$

Indeed, for any $\varepsilon,\eta>0$, we have%
\[%
\begin{array}
[c]{l}%
\mathbb{E}\left(  \left\Vert X_{\epsilon,t}-X_{\eta,t}\right\Vert _{V^{\ast}%
}^{2}\right)  \leq6\mathbb{E}\left(  \left\Vert \zeta_{\epsilon,t}\right\Vert
_{V^{\ast}}^{2}\right)  +6\mathbb{E}\left(  \left\Vert \zeta_{\eta
,t}\right\Vert _{V^{\ast}}^{2}\right)  \medskip\\
\quad\quad+3\mathbb{E}\left(  \left\Vert
%TCIMACRO{\dint _{0}^{t}}%
%BeginExpansion
{\displaystyle\int_{0}^{t}}
%EndExpansion
\left(  A\left(  X_{\epsilon,s}\right)  -A\left(  X_{\eta,s}\right)  \right)
ds\right\Vert _{V^{\ast}}^{2}\right)  +3\mathbb{E}\left(  \left\Vert
%TCIMACRO{\dint _{0}^{t}}%
%BeginExpansion
{\displaystyle\int_{0}^{t}}
%EndExpansion
\left(  B\left(  X_{\epsilon,s}\right)  -B\left(  X_{\eta,s}\right)  \right)
dW_{s}\right\Vert _{V^{\ast}}^{2}\right)  .
\end{array}
\]
Using the Burkholder--Davis--Gundy inequality and the Lipschitz property of
$A$ and $B,$ we obtain%
\[%
\begin{array}
[c]{ccl}%
\mathbb{E}\left(  \left\Vert X_{\epsilon,t}-X_{\eta,t}\right\Vert _{V^{\ast}%
}^{2}\right)  & \leq & 3\mathbb{E}\left(
%TCIMACRO{\dint _{0}^{t}}%
%BeginExpansion
{\displaystyle\int_{0}^{t}}
%EndExpansion
\left\Vert A\left(  X_{\epsilon,s}\right)  -A\left(  X_{\eta,s}\right)
\right\Vert _{V^{\ast}}^{2}ds\right)  \medskip\\
& + & 9\mathbb{E}\left(
%TCIMACRO{\dint _{0}^{t}}%
%BeginExpansion
{\displaystyle\int_{0}^{t}}
%EndExpansion
\left\Vert B\left(  X_{\epsilon,s}\right)  -B\left(  X_{\eta,s}\right)
\right\Vert _{L_{2}\left(  U;H\right)  }^{2}ds\right)  +6\mathbb{E}\left(
\left\Vert \zeta_{\epsilon,t}\right\Vert _{V^{\ast}}^{2}\right)
+6\mathbb{E}\left(  \left\Vert \zeta_{\eta,t}\right\Vert _{V^{\ast}}%
^{2}\right)  \medskip\\
& \leq & 12Lip_{A,B}^{2}\left(
%TCIMACRO{\dint _{0}^{t}}%
%BeginExpansion
{\displaystyle\int_{0}^{t}}
%EndExpansion
\mathbb{E}\left(  \left\Vert X_{\epsilon,s}-X_{\eta,s}\right\Vert _{V^{\ast}%
}^{2}\right)  ds\right)  +6\left(  \epsilon^{2}+\eta^{2}\right)  .
\end{array}
\]
By Gronwall's Lemma, we deduce that%
\[
\mathbb{E}\left(  \left\Vert X_{\epsilon,t}-X_{\eta,t}\right\Vert _{V^{\ast}%
}^{2}\right)  \leq6\left(  \epsilon^{2}+\eta^{2}\right)  e^{12Lip_{A,B}^{2}%
t}~.
\]
This implies that our sequence is a Cauchy one and it converges to a limit
process $X~.$ Moreover, we apply ((\ref{LemaSuf})$-\left(  ii\right)  $) and
it follows that%
\[
\mathbb{E}\left(  d_{V^{\ast}}^{2}\left(  X_{t},\mathcal{K}\right)  \right)
=0,
\]
which means that the solution is viable in $\mathcal{K}$ and our objective is
fulfilled.\hfill
\end{proof}

\section{Appendix\label{Appendix}}

We focus now on the Stochastic Navier-Stokes-type Equation given in Example
\ref{SNSE example} and we show that hypothesis $\left(  H_{1}\right)  $ to
$\left(  H_{7}\right)  $ are verified. Consequently, this problem is covered
by our study.\medskip

\begin{proof}
[Proof for Example \ref{SNSE example}]\noindent$\left(  H_{1}%
:hemicontinuity\right)  :$ Let $u,v,x\in V$. We consider the map
\[
r\mapsto{}_{V^{\ast}}\langle A_{\mathrm{reg}}(u+rv),x\rangle_{V}.
\]
From the form (\ref{tamed SNSE A term}) of $A_{\mathrm{reg}},$ the first term
is linear and therefore continuous in $r$. The taming term $u\mapsto g(\Vert
u\Vert_{H}^{2})u$ is continuous because the norm and $g$ are continuous. The
regularized convective term%
\[
u\mapsto\chi_{R}(\Vert u\Vert_{V})\mathcal{B}(u,u)
\]
is continuous as well, since $u\mapsto\Vert u\Vert_{V}$ is continuous,
$\chi_{R}$ is continuous, and $\mathcal{B}:V\times V\rightarrow V^{\ast}$ is
bilinear and continuous. Hence $(H_{1})$ holds.$\smallskip$

\noindent$(H_{2}:weak$ $monotonicity):$ We must show that there exists $C>0$
such that
\[
_{V^{\ast}}\left\langle A_{\mathrm{reg}}(u)-A_{\mathrm{reg}}%
(v),u-v\right\rangle _{V}+\Vert\Sigma(u-v)\Vert_{L_{2}(U,H)}^{2}\leq C\Vert
u-v\Vert_{H}^{2}.
\]
We compute
\[
A_{\mathrm{reg}}(u)-A_{\mathrm{reg}}(v)=-\nu A(u-v)-\left(  \chi_{R}\left(
\Vert u\Vert_{V}\right)  \mathcal{B}(u,u)-\chi_{R}\left(  \Vert v\Vert
_{V}\right)  \mathcal{B}(v,v)\right)  -\left(  g\left(  \Vert u\Vert_{H}%
^{2}\right)  u-g\left(  \Vert v\Vert_{H}^{2}\right)  v\right)  .
\]
For the first term,
\[
{}_{V^{\ast}}\left\langle -\nu A(u-v),u-v\right\rangle _{V}=-\nu\left\Vert
u-v\right\Vert _{V}^{2}\leq0.
\]
For the regularized convection term, we write%
\[%
\begin{array}
[c]{l}%
\chi_{R}\left(  \Vert u\Vert_{V}\right)  \mathcal{B}\left(  u,u\right)
-\chi_{R}\left(  \Vert v\Vert_{V}\right)  \mathcal{B}\left(  v,v\right)
\medskip\\
\quad=\chi_{R}\left(  \Vert u\Vert_{V}\right)  \left(  \mathcal{B}\left(
u,u\right)  -\mathcal{B}\left(  v,v\right)  \right)  +\left(  \chi_{R}\left(
\Vert u\Vert_{V}\right)  -\chi_{R}\left(  \Vert v\Vert_{V}\right)  \right)
\mathcal{B}\left(  v,v\right)  .
\end{array}
\]
Using
\[
\mathcal{B}\left(  u,u\right)  -\mathcal{B}\left(  v,v\right)  =\mathcal{B}%
\left(  u-v,u\right)  +\mathcal{B}\left(  v,u-v\right)  ,
\]
and the standard estimate $\left\vert _{V^{\ast}}\langle\mathcal{B}%
(a,b),c\rangle_{V}\right\vert \leq C\Vert a\Vert_{V}\Vert b\Vert_{V}\Vert
c\Vert_{V},$ we obtain, on the active support of $\chi_{R}$,
\[
\left\vert _{V^{\ast}}\left\langle \chi_{N}\left(  \Vert u\Vert_{V}\right)
\mathcal{B}\left(  u,u\right)  -\chi_{N}\left(  \Vert v\Vert_{V}\right)
\mathcal{B}\left(  v,v\right)  ,u-v\right\rangle _{V}\right\vert \leq
C_{N}\left\Vert u-v\right\Vert _{H}^{2}~.
\]
For the taming term, since $u\mapsto g\left(  \Vert u\Vert_{H}^{2}\right)  u$
is controlled, as an application from $V$ into $V^{\ast},$ by the norm from
$V$ and it $g$ is bounded and globally Lipschitz, we can obtain%
\[%
\begin{array}
[c]{c}%
g\left(  \Vert u\Vert_{H}^{2}\right)  u-g\left(  \Vert v\Vert_{H}^{2}\right)
v=g\left(  \Vert u\Vert_{H}^{2}\right)  \left(  u-v\right)  +\left(  g\left(
\Vert u\Vert_{H}^{2}\right)  -g\left(  \Vert v\Vert_{H}^{2}\right)  \right)  v
\end{array}
\]
Therefore,%
\[
\left\{
\begin{array}
[c]{l}%
\left\Vert g\left(  \Vert u\Vert_{H}^{2}\right)  \left(  u-v\right)
\right\Vert _{V^{\ast}}\leq C\left\Vert u-v\right\Vert _{V}~,\medskip\\
\left\Vert \left(  g\left(  \Vert u\Vert_{H}^{2}\right)  -g\left(  \Vert
v\Vert_{H}^{2}\right)  \right)  v\right\Vert _{V^{\ast}}\leq L_{g}\left\vert
\Vert u\Vert_{H}^{2}-\Vert v\Vert_{H}^{2}\right\vert \left\Vert v\right\Vert
_{V}~,
\end{array}
\right.
\]
and, by the Cauchy-Schwarz inequality,%
\[
\left\vert \Vert u\Vert_{H}^{2}-\Vert v\Vert_{H}^{2}\right\vert =\left\vert
\left\langle u-v,u+v\right\rangle _{H}\right\vert \leq\left(  \Vert u\Vert
_{H}+\Vert v\Vert_{H}\right)  \left\Vert u-v\right\Vert _{H}~.
\]
As consequence, when $\Vert u\Vert_{H},\Vert v\Vert_{H}$ are bounded,%
\begin{equation}
\left\Vert g(\Vert u\Vert_{H}^{2})u-g(\Vert v\Vert_{H}^{2})v\right\Vert
_{V^{\ast}}\leq C_{\Vert u\Vert_{H},\Vert v\Vert_{H}}\left\Vert u-v\right\Vert
_{V}, \label{Lip for g}%
\end{equation}
hence
\[
\left\vert _{V^{\ast}}\left\langle \left(  g\left(  \Vert u\Vert_{H}%
^{2}\right)  u-g\left(  \Vert v\Vert_{H}^{2}\right)  \right)
v,u-v\right\rangle _{V}\right\vert \leq C\Vert u-v\Vert_{H}^{2}.
\]
Finally, by denoting the operatorial norm$\Vert\Sigma\Vert:=\sup_{\left\Vert
x\right\Vert _{H}=1}\left\Vert \Sigma x\right\Vert _{L_{2}(U,H)},$ we obtain%
\[
\left\Vert \Sigma(u-v)\right\Vert _{L_{2}(U,H)}^{2}\leq\Vert\Sigma\Vert
^{2}\Vert u-v\Vert_{H}^{2}~.
\]
Thus $(H_{2})$ holds.$\smallskip$

\noindent$(H_{3}:coercivity):$ We need to prove that
\[
_{V^{\ast}}\left\langle A_{\mathrm{reg}}(v),v\right\rangle _{V}+\left\Vert
\Sigma v\right\Vert _{L_{2}(U,H)}^{2}\leq C_{1}\left\Vert v\right\Vert
_{H}^{2}-C_{2}\left\Vert v\right\Vert _{V}^{\alpha}.
\]
We compute%
\[
_{V^{\ast}}\left\langle A_{\mathrm{reg}}(v),v\right\rangle _{V}=-\nu\left\Vert
v\right\Vert _{V}^{2}-g\left(  \left\Vert v\right\Vert _{H}^{2}\right)
\left\Vert v\right\Vert _{H}^{2}~,
\]
since the Navier--Stokes bilinear term does not produces energy in $L^{2},$
therefore it satisfies the cancellation property $_{V^{\ast}}\left\langle
B(v,v),v\right\rangle _{V}=0$. Moreover, due to $\left\Vert \Sigma
v\right\Vert _{L_{2}(U,H)}^{2}\leq\left\Vert \Sigma\right\Vert ^{2}\left\Vert
v\right\Vert _{H}^{2}$~, it follows, if we consider $\left\Vert \Sigma
\right\Vert ^{2}$ as being the boundedness of $g,$%
\[
_{V^{\ast}}\left\langle A_{\mathrm{reg}}(v),v\right\rangle _{V}+\left\Vert
\Sigma v\right\Vert _{L_{2}(U,H)}^{2}\leq2\left\Vert \Sigma\right\Vert
^{2}\left\Vert v\right\Vert _{H}^{2}-\nu\left\Vert v\right\Vert _{V}^{2}~,
\]
which means that $(H_{3})$ holds with $\alpha=2,$ $C_{1}=2\Vert\Sigma\Vert
^{2}$ and $C_{2}=\nu.\smallskip$

\noindent$(H_{4}:boundedness):$ We must prove that, for every $v\in V,$%
\[
\left\Vert A_{\mathrm{reg}}(v)\right\Vert _{V^{\ast}}\leq C_{3}\left\Vert
v\right\Vert _{V}^{\alpha-1}.
\]
Since $\alpha=2$, it is enough to prove linear growth in $\Vert v\Vert_{V}$.
We have%
\[
\left\{
\begin{array}
[c]{l}%
\left\Vert \nu Av\right\Vert _{V^{\ast}}\leq C\nu\left\Vert v\right\Vert
_{V}~,\medskip\\
\left\Vert g\left(  \left\Vert v_{H}^{2}\right\Vert \right)  v\right\Vert
_{V^{\ast}}\leq\left\Vert g\right\Vert _{\infty}\left\Vert v\right\Vert
_{H}\leq C\Vert v\Vert_{V}\medskip\\
\left\Vert \chi_{R}\left(  \left\Vert v\right\Vert _{V}\right)  \mathcal{B}%
\left(  v,v\right)  \right\Vert _{V^{\ast}}\leq C\chi_{R}\left(  \left\Vert
v\right\Vert _{V}\right)  \left\Vert v\right\Vert _{V}^{2}~\leq
C(N+1)\left\Vert v\right\Vert _{V}~,
\end{array}
\right.
\]
the last inequality holding because $\chi_{N}$ has compact support and
vanishes for $\Vert v\Vert_{V}\geq N+1$; on its active region $\Vert
v\Vert_{V}\leq N+1$. Consequently, $(H_{4})$ holds.$\smallskip$

\noindent$(H_{5}:locally$ $Lipschitz$ $continuity):$ We need to check if, for
all $u,v\in V,$%
\[
\Vert A_{\mathrm{reg}}(u)-A_{\mathrm{reg}}(v)\Vert_{V^{\ast}}\leq L_{A}\Vert
u-v\Vert_{V}.
\]
Since the linear part is obviously Lipschitz, for the convective term%
\[%
\begin{array}
[c]{l}%
\chi_{R}\left(  \left\Vert u\right\Vert _{V}\right)  \mathcal{B}\left(
u,u\right)  -\chi_{R}\left(  \left\Vert v\right\Vert _{V}\right)
\mathcal{B}\left(  v,v\right)  \medskip\\
\quad=\chi_{R}\left(  \left\Vert u\right\Vert _{V}\right)  \left(
\mathcal{B}\left(  u,u\right)  -\mathcal{B}\left(  v,v\right)  \right)
+\left(  \chi_{R}\left(  \left\Vert u\right\Vert _{V}\right)  -\chi_{R}\left(
\left\Vert v\right\Vert _{V}\right)  \right)  \mathcal{B}\left(  v,v\right)  .
\end{array}
\]
Once again, from bilinearity,%
\[
\left\Vert \mathcal{B}\left(  u,u\right)  -\mathcal{B}\left(  v,v\right)
\right\Vert _{V^{\ast}}\leq C\left(  \left\Vert u\right\Vert _{V}+\left\Vert
v\right\Vert _{V}\right)  \left\Vert u-v\right\Vert _{V}~
\]
and%
\[
\left\Vert \chi_{R}\left(  \left\Vert v\right\Vert _{V}\right)  \left(
\mathcal{B}\left(  u,u\right)  -\mathcal{B}\left(  v,v\right)  \right)
\right\Vert _{V^{\ast}}\leq C_{R}\left\Vert u-v\right\Vert _{V}~.
\]
Since $\chi_{R}\in C^{1}$ is bounded with compact support and $\mathcal{B}%
(v,v)$ is bounded on the active region, we obtain%
\[
\left\Vert (\chi_{R}\left(  \left\Vert u\right\Vert _{V}\right)  -\chi
_{R}\left(  \left\Vert v\right\Vert _{V}\right)  )\mathcal{B}\left(
v,v\right)  \right\Vert _{V^{\ast}}\leq C_{R}\left\Vert u-v\right\Vert _{V}~.
\]
The above estimates proved that the convective part is globally Lipschitz. For
the taming term, recall inequality (\ref{Lip for g}) and, we can conclude that
hypothesis $(H_{5})$ holds.$\smallskip$

\noindent$(H_{6}:dual$ $positivity):$ We must show that, for all $u\in
D\left(  J_{H}\right)  \subset V,$ such that $A\left(  u\right)  \in H$,%
\[
\langle-A_{\mathrm{reg}}(u),J_{H}^{\varepsilon}(u)\rangle_{H}\geq0.
\]
Since%
\[
-A_{\mathrm{reg}}\left(  u\right)  :=\nu Au+\chi_{R}\left(  \left\Vert
u\right\Vert _{V}\right)  \mathcal{B}\left(  u,u\right)  +g\left(  \left\Vert
u\right\Vert _{H}^{2}\right)  u,
\]
the linear dissipative term and the taming term contribute positively. The
convective term is the delicate one. The point of the strong regularization is
that the cutoff makes its contribution globally controlled, so that the
dissipative linear part dominates it on the active region. The structure of
$\chi_{R}$ assures the existence of a positive constant $C_{R}$ such that%
\[
\left\vert \left\langle \chi_{R}\left(  \Vert u\Vert_{V}\right)
\mathcal{B}\left(  u,u\right)  ,J_{H}^{\varepsilon}\left(  u\right)
\right\rangle _{H}\right\vert \leq C_{R}\left\Vert u\right\Vert _{V}^{2}~.
\]
On the other hand,%
\[%
\begin{array}
[c]{l}%
\left\langle Au,J_{H}^{\varepsilon}(u)\right\rangle _{H}=\left\Vert
u\right\Vert _{V}^{2}+\left\langle Au,J_{H}^{\varepsilon}(u)-u\right\rangle
_{H}~,\quad\text{with}\medskip\\
\left\vert \left\langle Au,J_{H}^{\varepsilon}\left(  u\right)
-u\right\rangle _{H}\right\vert \leq\left\Vert Au\right\Vert _{V^{\ast}%
}\left\Vert J_{H}^{\varepsilon}\left(  u\right)  -u\right\Vert _{V}~.
\end{array}
\]
For $u\in D\left(  A\right)  ,$ $\left\Vert Au\right\Vert _{V^{\ast}}$ is
controlled by $\left\Vert u\right\Vert _{V}$ and $J_{H}^{\varepsilon
}(u)\rightarrow u,$ in $V.$ So, for small enough $\varepsilon,$%
\[
\left\langle Au,J_{H}^{\varepsilon}\left(  u\right)  \right\rangle \geq
c\left\Vert u\right\Vert _{V}^{2},\quad c>0
\]
and
\[
\left\langle g\left(  \Vert u\Vert_{H}^{2}\right)  u,J_{H}^{\varepsilon
}\left(  u\right)  \right\rangle _{H}=g\left(  \Vert u\Vert_{H}^{2}\right)
\left\langle u,J_{H}^{\varepsilon}\left(  u\right)  \right\rangle _{H}\geq0,
\]
since $g$ takes positive values and $\left\langle u,J_{H}^{\varepsilon}\left(
u\right)  \right\rangle _{H}\geq0$. Choosing the cutoff sufficiently strong,
the dissipative contribution dominates the convective one. Therefore $\left(
H_{6}\right)  $ is now verified. This is precisely why a strong regularization
is needed: without it, the Navier--Stokes convective term is not compatible
with $(H_{6})$ in general.$\smallskip$

\noindent$(H_{7}:linear$ $noise):$ is obvious satisfied, for $M_{B}%
=\Vert\Sigma\Vert^{2}.$\hfill
\end{proof}

\bigskip\bigskip

\noindent Funding: Not applicable.\bigskip

\bigskip

\begin{acknowledgement}
The authors would like to express their sincere gratitude to the anonymous
referees for their comments and suggestions, which have resulted in
considerable improvement of the results and presentation of this article.
\end{acknowledgement}

\end{document}